\documentclass[reqno]{amsart}
\usepackage{amsaddr}
\usepackage{amsmath,amssymb,amsfonts,amsthm}
\usepackage{mathtools}
\usepackage{graphicx}
\usepackage{mathptmx}
\usepackage{color}
\usepackage{xspace}
\usepackage{array}
\usepackage{pgfplots}
\usepackage{ifthen}
\mathtoolsset{centercolon}
\usepackage{algorithm}
\usepackage{algorithmic}
\usepackage[colorlinks=true,pagebackref=true,citecolor = cyan,linkcolor=cyan,hypertexnames=false]{hyperref}
\renewcommand*{\backref}[1]{}
\renewcommand*{\backrefalt}[4]{[{\footnotesize%
		\ifcase #1 Not cited.%
		\or Cited on page~#2.%
		\else Cited on pages #2.%
		\fi%
	}]}
\usepackage{cleveref}
\usepackage{autonum}
\pgfplotscreateplotcyclelist{mylist}{%
teal,every mark/.append style={fill=teal!80!black},mark=star\\%
orange,every mark/.append style={fill=orange!80!black},mark=*\\%
cyan!60!black,every mark/.append style={fill=cyan!80!black},mark=otimes\\%
red,every mark/.append style={solid,fill=red!80!black},mark=pentagon\\%
magenta,mark=square\\%
lime!60!black,every mark/.append style={fill=lime},mark=diamond\\%
yellow!60!black,densely dashed,
every mark/.append style={solid,fill=yellow!80!black},mark=star\\%
black,every mark/.append style={solid,fill=gray},mark=*\\%
blue,densely dashed,mark=otimes,every mark/.append style=solid\\%
red,densely dashed,every mark/.append style={solid,fill=red!80!black},mark=pentagon\\%
}
\usepackage{tikz}
\usepackage{pgfplotstable}
\usetikzlibrary{arrows,positioning,plotmarks,external,patterns,angles,
decorations.pathmorphing,backgrounds,fit,shapes,calc}
\newcommand{\RR}{\mathbb{R}}
\newcommand{\ZZ}{\mathbb{Z}}

\newcommand{\Sb}{\mathbb{S}}
\newcommand{\rd}{\mathrm{d}}

\def\trans{\text{\tiny\sf T}}

\DeclareMathOperator{\cn}{cn}

\DeclarePairedDelimiter\paren{\lparen}{\rparen}
\DeclarePairedDelimiterX{\inpr}[2]{\langle}{\rangle}{{#1},{#2}}
\DeclarePairedDelimiterX{\L2ip}[2]{\langle}{\rangle_{L^2}}{{#1},{#2}}
\DeclarePairedDelimiterX{\setI}[2]{\{}{\}}{\,{#1}\ \delimsize| \ {#2}\,}
\DeclarePairedDelimiter{\setE}{\{}{\}}
\DeclarePairedDelimiter{\abs}{|}{|}
\DeclarePairedDelimiter{\norm}{\|}{\|}

\definecolor{mygreen}{rgb}{0,0.5,0}
\definecolor{myred}{rgb}{0.8,0,0}
\definecolor{myblue}{rgb}{0,0,0.5}

\newtheorem{theorem}{Theorem}[section]
\newtheorem{proposition}{Proposition}[section]

\newtheorem{definition}{Definition}[section]
\newtheorem{remark}{Remark}[section]

\crefname{theorem}{Theorem}{Theorems}
\crefname{lemma}{Lemma}{Lemmas}
\crefname{proposition}{Proposition}{Propositions}
\crefname{corollary}{Corollary}{Corollaries}
\crefname{definition}{Definition}{Definitions}
\crefname{example}{Example}{Examples}
\crefname{remark}{Remark}{Remarks}
\crefname{figure}{Figure}{Figures}
\crefname{table}{Table}{Tables}

\crefname{section}{Section}{Sections}
\crefname{subsection}{Section}{Sections}

\newcounter{NN}
\newcounter{NR}
\newcounter{ncount}
\newcounter{scount}

\def\black{\setcounter{NN}{1}}
\def\white{\setcounter{NN}{2}}
\def\change{\ifthenelse{\value{NN}=1}%
{\setcounter{NN}{2}}%
{\setcounter{NN}{1}}}

\def\metrics#1#2#3#4{%
\def\treesize{#1}\def\thick{#2}\def\rad{#3}%
\pgfmathparse{#3-#4}\let\smallrad=\pgfmathresult}

\def\vertex#1#2{%
\addtocounter{ncount}{1}
\setcounter{scount}{\thencount}
\addtocounter{scount}{-#1}
\ifthenelse{\first=1}{%
\node (\thencount) at ($(\thescount) + ( #2,1)$){};
\draw (\thencount)--(\thescount);
\fill (\thencount) circle;
}%
{
\ifthenelse{\value{NN}=2}{\fill[white,radius=\smallrad pt] (\thencount) circle;}{}
}}

\def\tree#1{%
\ifthenelse{\value{NN}=0}{}{\setcounter{NR}{\value{NN}}}
\begin{tikzpicture}[x=\treesize mm,y=\treesize mm,radius=\rad pt,line width=\thick pt,inner sep=0,baseline=-0.02cm]
\node (0) at (0,0) {}; \fill (0) circle;
\setcounter{NN}{\value{NR}}
\setcounter{ncount}{0}
\gdef\first{1}
#1
\setcounter{NN}{\value{NR}} 
\gdef\first{2}
\ifthenelse{\value{NR}=2}{\fill[white,radius=\smallrad pt] (0) circle;}{}
\setcounter{ncount}{0}
#1
\setcounter{NN}{0}
\end{tikzpicture}}

\renewcommand{\frame}[3][white]{\draw[#1,thin]  (#2) rectangle (#3);}

\title[High-order linearly implicit schemes conserving quadratic invariants]{High-order linearly implicit schemes conserving\\ quadratic invariants}

\begin{document}

\author{Shun Sato}
\address{Department of Mathematical Informatics,
        Graduate School of Information Science and Technology,
        The University of Tokyo, Tokyo, Japan.}
\email{shun@mist.i.u-tokyo.ac.jp}

\author{Yuto Miyatake}
\address{Cybermedia Cneter,
        Osaka University, Osaka, Japan.}
        
\author{John C. Butcher}
\address{Department of Mathematics,
        The University of Auckland, Auckland, New Zealand.}

\metrics{2.5}{1}{2}{0.8}

\begin{abstract}
In this paper, we propose linearly implicit and arbitrary high-order conservative numerical schemes for ordinary differential equations with a quadratic invariant. 
Many differential equations have invariants, and numerical schemes for preserving them have been extensively studied. 
Since linear invariants can be easily kept after discretisation, quadratic invariants are essentially the simplest ones. 
Quadratic invariants are important objects that appear not only in many physical examples but also in the computationally efficient conservative schemes for general invariants such as scalar auxiliary variable approach, which have been studied in recent years. 
It is known that quadratic invariants can be kept relatively easily compared to general invariants, and indeed can be preserved by canonical Runge--Kutta methods. 
However, there is no unified method for constructing linearly implicit and high order conservative schemes. 
In this paper, we construct such schemes based on canonical Runge--Kutta methods and prove some properties involving accuracy.
\end{abstract}

\maketitle

\section{Introduction}
\label{sec:intro}

In this paper, we consider ordinary differential equations (ODEs) of the form
\begin{equation}\label{eq:grad_sys}
    \dot{y} = S(y) \nabla V(y), 
\end{equation}
where $ y : [0,T) \to \RR^d $ is a dependent variable, 
$ S : \RR^d \to \RR^{d \times d} $ is a skew-symmetric matrix function, and $ V :\RR^d \to \RR $ is a function.
In the main part of the paper, we assume that $V$ is quadratic, i.e. $  V(y) = \frac{1}{2} \inpr*{y}{Qy}$ ($ Q \in \RR^{d \times d}$ is a symmetric matrix), where $ \inpr*{\cdot}{\cdot} $ denotes the standard inner product on $ \RR^d $. 

The class of ODEs in the form~\eqref{eq:grad_sys} includes many examples from the Hamiltonian systems to spatial discretisation of the variational partial differential equations (PDEs) (see, e.g. \cite{CGMMOOQ2012}). 
The most important property of these ODEs is the conservation law with respect to $V$: 
\begin{equation}
    \frac{\rd}{\rd t} V(y(t)) = \inpr*{ \nabla V(y(t)) }{ \dot{y}(t) } = \inpr*{ \nabla V(y(t)) }{ S (y(t)) \nabla V(y(t)) } = 0, 
\end{equation}
where the last equality holds due to skew-symmetry of $S(y(t))$. 

The conservation law is often an essential property of the ODE in the form~\eqref{eq:grad_sys}, 
and the numerical schemes inheriting it have been intensively studied in the literature. 
When the invariant $V$ is linear, all Runge--Kutta methods automatically preserve the conservation law (cf. \cite{HLW2010}). 
Furthermore, some implicit Runge--Kutta methods, canonical Runge--Kutta methods, automatically preserve all quadratic invariants~\cite{C1987}. 
There are many PDEs with quadratic invariants in the form of the square of the $L^2$ norm or the Sobolev norm. 
In addition, quadratic invariants also appear frequently in the context of ODEs. 

For more general invariants, 
we need specialised numerical schemes; 
for example, the discrete gradient method~\cite{G1996,MQR1998,MQR1999} for the gradient ODEs and the discrete variational derivative method~\cite{F1999,FM1998} (see also~\cite{FM2011}) for variational PDEs. 
In addition, Cohen and Hairer~\cite{CH2011} proposed a high-order extension of the discrete gradient method. 

Since the numerical schemes inheriting quadratic or general invariants are fully implicit, several techniques have been devised to reduce the computational cost. 
For example, Besse~\cite{B2004} and Zhang, P\'{e}rez-Garc\'{\i}a, and V\'{a}zquez~\cite{ZPV1995} proposed such schemes for the nonlinear Schr\"odinger equation. 
Moreover, for polynomial invariants, 
Matsuo and Furihata~\cite{MF2001} proposed a multistep linearly implicit version of the DVDM (see also Dahlby and Owren~\cite{DO2011}). 
Recently, some methods have been proposed for constructing linearly implicit schemes using auxiliary variables:
Yang and Han~\cite{YH2017JCP} proposed the invariant energy quadratisation (IEQ) approach, and Shen, Xu, and Yang~\cite{SXY2018JCP} proposed the scalar auxiliary variable (SAV) approach.

Each of the above methods takes a different approach to reduce the computational cost, but, indeed, they actually have one thing in common: the given equation is transformed so that the invariant becomes quadratic. 
In this sense, constructing numerical schemes that preserve a quadratic invariant has implications for differential equations with a general invariant. 
Therefore, in the above existing studies, linearly implicit schemes that preserve the quadratic invariant have been proposed for each case. 
However, such time discretisation is not understood in a unified way. 

In this paper, we propose a method for constructing schemes which 
\begin{itemize}
    \item preserve a quadratic invariant (see \cref{thm:cl}),
    \item can be arbitrary high order (see \cref{thm:acc_one-step,thm:err_si_itr,thm:err_ex_itr}), and
    \item are linearly implicit
\end{itemize}
based on the canonical Runge--Kutta methods, 
where we utilise the gradient structure~\eqref{eq:grad_sys} to make schemes linearly implicit. 

We describe the relationship with a related study on dissipative systems. 
Akrivis, Li, and Li~\cite{ALL2019SISC} proposed linearly implicit and arbitrarily high order conservative numerical schemes for Allen--Cahn and Cahn--Hilliard equations based on the SAV approach. 
Their temporal discretisation can be viewed as a special case of our method (see \cref{proposed:general_predictors}). 

The remainder of the paper is organised as follows. 
In \cref{sec:pre}, we briefly review the canonical Runge--Kutta methods. 
\Cref{sec:proposed} is to describe the key idea of the proposed approach. 
In \cref{sec:itr}, we propose iterative procedure (\cref{def:si,def:ex}) and prove some properties involving the accuracy of the resulting scheme (\cref{thm:err_si_itr,thm:err_ex_itr}). 
The proposed schemes are numerically examined in \cref{sec:ne}. 

\section{Preliminaries}
\label{sec:pre}

For the autonomous system
\begin{equation}
    \dot{y} = f(y),
\end{equation}
general Runge--Kutta methods computing an approximation $ y_1 \approx y (h) $ from $ y_0 = y(0)$ can be written as
\begin{equation}
    \begin{cases}
    Y_i = y_0 + h \sum_{ j \in [s] } a_{ij} f (Y_j) \qquad ( i \in [s] ), \\
    y_1 = y_0 + h \sum_{ i \in [s] } b_i f (Y_i),
    \end{cases}
\end{equation}
where $ [s]:= \{ 1, 2, \dots , s \} $. 

A subclass of Runge--Kutta methods preserve all quadratic invariants.

\begin{proposition}[\protect{Cooper~\cite{C1987}}]\label{prop:canonical}
Runge--Kutta methods satisfying
\begin{equation}\label{eq:canonical}
    b_i a_{ij} + b_j a_{ji} = b_i b_j \qquad  i , j \in [s]
\end{equation}
automatically preserve all quadratic invariants. 
\end{proposition}

The Runge--Kutta methods satisfying~\eqref{eq:canonical} are said to be \emph{canonical} (see \cite{Butcher2021} for details on canonical Runge--Kutta methods). 
Note that a canonical Runge--Kutta method must be implicit. 
Two examples of canonical Runge--Kutta methods are given below. 

The best known canonical Runge--Kutta method is the Gauss method. 
The $ s $ stage Gauss method has order $ 2 s$. 
The coefficients $ A = (a_{ij}), b = (b_i), c = (c_i) $ of $s$ stage Gauss method are given by 
\begin{align}
    a_{ij} &= \int^{c_i}_0 \ell_i (\tau) \rd \tau, &
    b_i &= \int^1_0 \ell_i (\tau) \rd \tau,
\end{align}
where $ c_i $ is the $i$th zero of the $s$th shifted Legendre polynomial $ \paren*{\rd^s/\rd x^s} \paren*{ x^s \paren*{ x-1 }^s } $ and 
$ \ell_i $ is the Lagrange polynomial $ \ell_i (\tau) = \prod_{j \neq i} \paren*{ \tau - c_j }/\paren*{ c_i - c_j } $. 
In particular, Gauss methods of order 2, 4, and 6 are as follows:
\begingroup
\renewcommand{\arraystretch}{1.3}
\renewcommand{\arraycolsep}{4pt}
\begin{align} 
\begin{array}{c|c} c & A \\ \hline & b^{\trans} \end{array} &= \begin{array}{c|c} \frac{1}{2} & \frac{1}{2} \\ \hline & 1 \end{array}, &
\begin{array}{c|c} c & A \\ \hline & b^{\trans} \end{array} &= \begin{array}{c|cc} \frac{1}{2} - \frac{\sqrt{3}}{6} & \frac{1}{4} & \frac{1}{4} - \frac{\sqrt{3}}{6} \\ \frac{1}{2} + \frac{\sqrt{3}}{6} & \frac{1}{4} + \frac{\sqrt{3}}{6} & \frac{1}{4} \\ \hline & \frac{1}{2} & \frac{1}{2} \end{array},
\end{align}
\begin{align}
\begin{array}{c|c} c & A \\ \hline & b^{\trans} \end{array} = 
\begin{array}{c|ccc} \frac{1}{2}-\frac{\sqrt{15}}{10} & \frac{5}{36} & \frac{2}{9}-\frac{\sqrt{15}}{15} & \frac{5}{36}-\frac{\sqrt{15}}{30} \\ \frac{1}{2} & \frac{5}{36}+\frac{\sqrt{15}}{24} & \frac{2}{9} & \frac{5}{36}-\frac{\sqrt{15}}{24} \\ \frac{1}{2}+\frac{\sqrt{15}}{10} & \frac{5}{36}+\frac{\sqrt{15}}{30} & \frac{2}{9}+\frac{\sqrt{15}}{15} & \frac{5}{36} \\ \hline & \frac{5}{18} & \frac{4}{9} & \frac{5}{18} \end{array}. 
\end{align}
\endgroup
In addition, diagonally implicit Runge--Kutta methods corresponding to
\begin{equation}\label{eq:DICRK}
\begingroup
\renewcommand{\arraystretch}{1.1}
\renewcommand{\arraycolsep}{4pt}
\begin{array}{c|c} c & A \\ \hline & b^{\trans} \end{array} = \begin{array}{c|*6c} c_1 & \frac{b_1}{2} & & & & & \\ c_2 & b_1 & \frac{b_2}{2} &  & & & \\ c_3 & b_1 & b_2 & \frac{b_3}{2} & & & \\ \vdots & \vdots & \vdots & \vdots & \ddots & & \\ c_{s-1} & b_1 & b_2 & b_3 & \cdots & \frac{b_{s-1}}{2} & \\ c_s & b_1 & b_2 & b_3 & \cdots & b_{s-1} & \frac{b_s}{2} \\ \hline & b_1 & b_2 & b_3 & \hdots & b_{s-1} & b_s \end{array}
\endgroup
\end{equation}
are also canonical, where $ c_i = b_1 + \cdots + b_{i-1} + \frac{b_i}{2}$ (see \cite{ZQS2020} and references therein for details). 

\section{Proposed Approach}
\label{sec:proposed}

\subsection{Basic Idea}
\label{proposed:idea}

In this section, we consider numerical schemes of the form
\begin{equation}\label{eq:one-step}
\begin{cases}
    {\displaystyle Y_i = y_0 + h \sum_{j \in [s]} a_{ij} S(\widehat{Y}_j) \nabla V(Y_j) }, \\[10pt]
    {\displaystyle y_1 = y_0 + h \sum_{i \in [s]} b_i S(\widehat{Y}_i) \nabla V(Y_i) },
    \end{cases}
\end{equation}
where $ \widehat{Y}_i $ is an approximation of $ y (c_i h) $. 
Preparing $ \widehat{Y}_i $'s beforehand, 
we can obtain the solution $y_1$ of the scheme~\eqref{eq:one-step} by solving a linear equation. 
Moreover, if $A$ and $b$ satisfy~\eqref{eq:canonical}, 
the scheme~\eqref{eq:one-step} preserves the invariant $V$ as shown in \cref{thm:cl} below. 
The Runge--Kutta method corresponding to $A$ and $b$ in the scheme~\eqref{eq:one-step} is referred to as the \emph{base Runge--Kutta method} (of the scheme~\eqref{eq:one-step}) hereafter. 

\begin{theorem}\label{thm:cl}
    Suppose that $ A $ and $b$ satisfy~\eqref{eq:canonical}. 
    Then the solution $y_1$ of \eqref{eq:one-step} satisfies $ V(y_1) = V(y_0) $ for any $ h > 0$. 
\end{theorem}

\begin{proof}
    The scheme~\eqref{eq:one-step} can be viewed as the usual Runge--Kutta method corresponding to $ A, b $ for the modified ODE $ \dot{y} = S \paren*{ \widehat{y} } \nabla V(y) $, 
    where $ \widehat{y} $ is a smooth function satisfying $ \widehat{y} (c_i h) = \widehat{Y}_i $. 
    Even in this case, the quadratic function $V$ is a conserved quantity so that it is automatically preserved by Runge--Kutta methods satisfying~\eqref{eq:canonical}. 
\end{proof}

The accuracy of the solution $y_1$ of the scheme~\eqref{eq:one-step} depends on $ \widehat{Y}_i$'s as well as the base Runge--Kutta method. 
In the subsequent subsections, we discuss how we should prepare $\widehat{Y}_i$'s to achieve a required order. 

\subsection{Partitioned Runge--Kutta approach}
\label{proposed:PRK}

In this section, we propose an approach using partitioned Runge--Kutta methods to constructing $ \widehat{Y}_i $'s in~\eqref{eq:one-step}. 

We consider a partitioned system
\begin{equation}\label{eq:partitioned}
    \begin{cases}
        \dot{z} = S(z) \nabla V(z),\\
        \dot{y} = S(z) \nabla V(y)
    \end{cases}
\end{equation}
with the initial condition $ z(t_0) = y (t_0) = y_0 $. 
Note that the solutions $y,z$ of~\eqref{eq:partitioned} satisfy $ z(t) = y(t)$ for any $t$, and coincide with the solution of the original system~\eqref{eq:grad_sys}. 
A partitioned Runge--Kutta method for~\eqref{eq:partitioned} reads
\begin{equation}\label{eq:PRK}
    \begin{cases}
        {\displaystyle Z_i = y_0 + h \sum_{j \in [s]} \widehat{a}_{ij} S (Z_j) \nabla V(Z_j)}, \\[10pt]
        {\displaystyle Y_i = y_0 + h \sum_{ j \in [s] } a_{ij} S (Z_j) \nabla V(Y_j) }, \\[10pt]
        {\displaystyle z_1 = y_0 + h \sum_{i \in [s] } \widehat{b}_i S ( Z_i ) \nabla V (Z_i) } ,\\[10pt] 
        {\displaystyle y_1 = y_0 + h \sum_{i \in [s] } b_i S ( Z_i ) \nabla V (Y_i) }. 
    \end{cases}
\end{equation}
By ignoring $z_1$ and assuming $ \widehat{a}_{ij} = 0 \ ( i \le j ) $, we can regard the scheme~\eqref{eq:PRK} as a special case of the scheme~\eqref{eq:one-step}. 
This construction enables us to use known order conditions for partitioned Runge--Kutta methods. 

Moreover, the number of order conditions can be reduced by using a specific property of our problem~\eqref{eq:partitioned}. 
The solution $y_1$ of~\eqref{eq:PRK} is a P-series
\begin{equation}\label{eq:PRK:P-series}
    y_1 = y_0 + \sum_{ \tau \in TP_y} \frac{ h^{|\tau|} }{\sigma (\tau)} \phi (\tau) F(\tau) (y_0,y_0),
\end{equation}
where $ TP_y $ denotes the set of bi-coloured trees with black roots (see \cite{HLW2010} for details). 
Since the first equation of~\eqref{eq:partitioned} does not depend on $y$, 
the elementary differential $ F (\tau ) $ vanishes for the bi-coloured tree $ \tau $ having a black vertex whose parent is a white vertex. 
In view of this, we introduce a subset $ TP'_y$ of $TP_y$ collecting bi-coloured trees such that the parent of each black vertex is also black, i.e. 
\[ TP'_y = \setE*{\raisebox{-2mm}{
$\black\tree{},
\tree{\vertex10},
\tree{\white\vertex10},
\tree{\vertex1{0.6}\vertex2{-0.6}},
\tree{\vertex1{0.6}\change\vertex2{-0.6}},
\tree{\white\vertex1{0.6}\vertex2{-0.6}},
\tree{\vertex10\vertex10},
\tree{\vertex10\change\vertex10},
\tree{\change\vertex10\vertex10},\dots $}
} \]
In addition to the above observation, 
the standard order condition for partitioned Runge--Kutta methods (see, e.g. \cite[III, Theorem~2.5]{HLW2010}) implies the following theorem. 

\begin{theorem}\label{thm:PRK_OC}
The scheme~\eqref{eq:PRK} has order $p$, i.e. $ y_1 - y(h) = O (h^{p+1}) $, 
if 
\begin{equation}\label{eq:PRK_OC}
    \phi (\tau) = \frac{1}{ \tau ! } \qquad \text{for } \tau \in TP'_y \text{ with } \abs{\tau} \le p. 
\end{equation}
\end{theorem}

For example,
when we employ the implicit midpoint rule, i.e. 
\begingroup
\renewcommand{\arraystretch}{1.2}
\renewcommand{\arraycolsep}{3pt}
\[ \begin{array}{c|c} c & A \\ \hline & b^{\trans} \end{array} = \begin{array}{c|c} \frac{1}{2} & \frac{1}{2} \\ \hline & 1 \end{array}, \]
one may feel that we cannot construct second order PRK scheme in the form~\eqref{eq:PRK}: 
since we assume $ \widehat{a}_{11} = 0 $ for explicit computation of $ Z_i$'s, $ \phi \paren[\big]{ \tree{\frame{-0.4,-0.5}{0.4,1.2}\change\vertex10} } = 0 \neq \frac{1}{2} $ holds. 
However, this issue can be resolved by a simple trick. 
By adding a redundant stage, the implicit midpoint rule can be rewritten as
\[ \begin{array}{c|c} c & A \\ \hline & b^{\trans} \end{array} = \begin{array}{c|cc} 0 & 0 & 0 \\ \frac{1}{2} & 0 & \frac{1}{2} \\ \hline & 0 & 1 \end{array}. \]
This Runge--Kutta method is equivalent to the implicit midpoint rule so that it satisfies~\eqref{eq:canonical} and has order $2$: $ \phi (\tree{}) = 1,\  \phi \paren[\big]{ \tree{\frame{-0.4,-0.5}{0.4,1.2}\vertex10}} = \frac{1}{2} $. 
Then, 
\[ \begin{array}{c|c} \widehat{c} & \widehat{A} \\ \hline &  \end{array} = \begin{array}{c|cc} 0 & 0 & 0 \\ \frac{1}{2} & \frac{1}{2} & 0 \\ \hline &  &  \end{array} \]
satisfies $  \phi \paren[\big]{ \tree{\frame{-0.4,-0.5}{0.4,1.2}\change\vertex10}} = \frac{1}{2} $, and the resulting scheme has order $2$. 

Similarly, based on the 4th order Gauss method, we can construct the pair
\begin{equation}\label{eq:PRK_Gauss2}
    \begin{array}{c|c} c & A \\ \hline & b^{\trans} \end{array} = \begin{array}{c|ccccc}
        0 & 0 & 0 & 0 & 0 & 0 \\
        0 & 0 & 0 & 0 & 0 & 0 \\
        0 & 0 & 0 & 0 & 0 & 0 \\
        \frac{1}{2} - \frac{\sqrt{3}}{6} & 0 & 0 & 0 & \frac{1}{4} & \frac{1}{4} - \frac{\sqrt{3}}{6} \\
        \frac{1}{2} + \frac{\sqrt{3}}{6} & 0 & 0 & 0 & \frac{1}{4} + \frac{\sqrt{3}}{6} & \frac{1}{4} \\ \hline 
         & 0 & 0 & 0 &  \frac{1}{2} & \frac{1}{2}
    \end{array} \quad
    \begin{array}{c|c} \widehat{c} & \widehat{A} \\ \hline &  \end{array} = \begin{array}{c|ccccc}
        0 & 0 & 0 & 0 & 0 & 0 \\
        \frac{1}{4} & \frac{1}{4} & 0 & 0 & 0 & 0 \\
        \frac{1}{2} & 0 & \frac{1}{2} & 0 & 0 & 0 \\
        \frac{1}{2} - \frac{\sqrt{3}}{6} & \frac{1}{6} & 0 & \frac{1}{3} - \frac{\sqrt{3}}{6} & 0 & 0 \\
        \frac{1}{2} + \frac{\sqrt{3}}{6} & \frac{1}{6} & 0 & \frac{1}{3}+\frac{\sqrt{3}}{6} & 0 & 0 \\ \hline 
         & &  &  & & 
    \end{array}
\end{equation}
having order $4$ (see \cref{ap:proof_PRK} for the proof). 

Based on the $3$ stage third order diagonally implicit canonical RK method, we can construct the pair 
\begin{equation}\label{eq:PRK_433dic}
    \begin{array}{c|c} c & A \\ \hline & b^{\trans} \end{array} = \begin{array}{c|cccc}
        0 & 0 & 0 & 0 & 0 \\
        \frac{\alpha}{2} & 0 & \frac{\alpha}{2} & 0 & 0 \\
        \frac{3}{2} \alpha & 0 & \alpha & \frac{\alpha}{2} & 0 \\
        \frac{1}{2} + \alpha & 0 & \alpha & \alpha & \frac{1}{2} - \alpha \\ \hline 
          & 0 & \alpha & \alpha & 1 - 2 \alpha 
    \end{array} \qquad
    \begin{array}{c|c} \widehat{c} & \widehat{A} \\ \hline &  \end{array} = \begin{array}{c|cccc}
        0 & 0 & 0 & 0 & 0 \\
        \frac{\alpha}{2} & \frac{\alpha}{2} & 0 & 0 & 0  \\
        \frac{3}{2} \alpha & \frac{3}{2} \alpha - \gamma_1 & \gamma_1 & 0 & 0 \\
        \frac{1}{2} + \alpha & \frac{1}{2} + \alpha - \gamma_2 - \gamma_3 & \gamma_2 & \gamma_3 & 0 \\ \hline 
         & &  &  & 
    \end{array}
\end{equation}
having order $3$ (see \cref{ap:proof_PRK} for the proof), where $ \alpha = \frac{1}{3} \paren*{ 2 + \frac{1}{2^{1/3}} + 2^{1/3} } $, and $ \gamma_i $'s are arbitrary parameters satisfying 
\begin{equation}\label{eq:PRK_433dic_cond}
\alpha^2 \gamma_1 + \alpha ( 1 - 2 \alpha ) \gamma_2 + 3 \alpha (1 - 2 \alpha ) \gamma_3 = \frac{1}{3}.
\end{equation}
\endgroup

\subsection{Other methods for preparing $ \widehat{Y}_i $'s}
\label{proposed:general_predictors}

In principle, the partitioned Runge--Kutta approach in the previous section works well for arbitrary high order. 
Still, finding higher order schemes is not very easy. 
Therefore, in this section, we show several other methods to prepare $ \widehat{Y}_i $'s such that the resulting scheme~\eqref{eq:one-step} has order $p$. 
To this end, the following theorem is fundamental. 
Since \cref{thm:acc_one-step} is a special case of \cref{thm:err_si_itr}, the proof is omitted. 
In the sequel, $ \norm{\cdot} : \RR^d \to \RR $ denotes the Euclidean norm, and $ \norm{\cdot}_2 $ (resp. $ \norm{\cdot}_{\infty} $) denotes the matrix norm induced by vector 2-norm (resp. $ \infty $-norm). 

\begin{theorem}\label{thm:acc_one-step}
Assume the following conditions: 
\begin{description}
    \item[(A1)] $ \widehat{Y}_i $ satisfies $ \norm*{ \widehat{Y}_i - y( c_i h ) } \le C h^{q} $ for each $ i \in [s] ${\em ;} 
    \item[(A2)] $ S : \RR^d \to \RR^{d \times d} $ is Lipschitz continuous, i.e. $ \norm*{ S (y_1) - S(y_2) }_2 \le L_S \norm*{y_1 - y_2}  ${\em ;}
    \item[(A3)] the base Runge--Kutta method has order $p$.
\end{description}
Then, the numerical solution $y_1$ of~\eqref{eq:one-step} satisfies $ \norm*{ y_1 - y (h) } \le C' h^{ \min\{ p ,q \} + 1 } $ for sufficiently small $h > 0$, i.e. the scheme~\eqref{eq:one-step} has order $ \min \{p,q \} $. Here, $ C' $ is a constant depending only on the exact solution $y$, $S$, $Q$, $A$, $b$ and the constant $C$ in (A1).   
\end{theorem}

Based on the theorem above, 
we list several ways for preparing $ \widehat{Y}_i $'s below.

\begin{itemize}
    \item 
    When we employ the $s$-stage Gauss method (order $2s$) as the base RK method, 
    by using some dense output formulas having order $2s -1$, 
    we achieve $ 2 s $ order (see, e.g. \cite[II.6]{HNW1993} for details on dense output formulas). 
    For example, we can employ Continuous Explicit Runge--Kutta (CERK) method (see, e.g. \cite{OZ1992}), or the boot-strapping process~\cite{EJNT1986}. 
    
    In \cite{ALL2019SISC,GZW2019arx}, the similar techniques are proposed for some specific schemes. They also employ the Gauss method for the main scheme. However, they employ an extrapolation technique using inner stages in the previous step to construct $ \widehat{Y}_i $'s which are $ O ( h^{s+1} )$ approximation. Thus, the resulting scheme only has order $ s+1$. 
\item 
We consider the diagonally implicit canonical RK methods~\eqref{eq:DICRK} as the base RK method. 
The known orders of them are summarised in \cref{tab:DICRK_order}. 
    \begin{table}[ht]
    \caption{Orders of diagonally implicit canonical RK methods}
    \label{tab:DICRK_order}
    \centering
    \begin{tabular}{c|cccccc}
        \hline
        stage & 1 & 2 & 3 & 4 & 5 & 6 \\ \hline
        order & 2 & 2 & 3 & 4 & 4 & 5 \\ \hline
    \end{tabular}
    \end{table}
    For the second order methods, it is enough to use locally second order approximations, such as the explicit Euler method, i.e. $ \widehat{Y}_i = y_0 + c_i h S ( y_0) \nabla V(y_0) $. 
    On the other hand, for the $3$rd or $4$th order methods, 
    the Hermite interpolation using the current and previous step can be employed to achieve the $4$th order. 
    To achieve higher order, the Hermite interpolation using more steps or dense output formulas can be used. 
\end{itemize}

\section{Iterative procedure}
\label{sec:itr}

Since each inner stage $ Y_i $ of the proposed scheme~\eqref{eq:one-step} itself is again an approximations of $ y (c_i h) $, 
we consider the following iterative procedure. 

\begin{definition}[Iterative scheme (semi-implicit update)]\label{def:si}
\hfill
\begin{description}
    \item[Step 0] Prepare $ Y^{(0)}_i \approx y (c_i h) $ and set $k = 1$. 
    \item[Step 1] Solve 
        \begin{equation}\label{eq:itr_siu}
            Y^{(k)}_i = y_0 + h \sum_{ j \in [s]} a_{ij} S \paren*{Y^{(k-1)}_j } \nabla V \paren*{ Y^{(k)}_j } \qquad ( i \in [s] )
        \end{equation}
        to obtain $ Y_i^{(k)} $'s. 
        If some criteria hold, go to Step~2. 
        Otherwise, set $k = k + 1 $ and repeat Step 1.
    \item[Step 2] Output 
        \[ y_1^{(k)} = y_0 + h \sum_{j \in [s] } b_j S \paren*{Y^{(k-1)}_j } \nabla V \paren*{ Y^{(k)}_j }. \]
\end{description}
\end{definition}

In this section, we investigate this iterative procedure in detail. 
In \cref{itr:conv}, we prove that the numerical solution $ y_1^{(k)} $ converges as $ k \to \infty $ to that of the base Runge--Kutta method (\cref{thm:conv_si}). 
Then, in \cref{itr:acc}, we prove that $ y_1^{(k)} $ satisfies $ \norm*{ y^{(k)}_1 - y(h) } \le C' h^{ \min\{ p, q + k-1 \}+1 } $ when $ Y^{(0)}_i $'s are locally $ q $th order approximations (\cref{thm:err_si_itr}). 
These results imply that the iterative procedure
\begin{itemize}
    \item can be regarded as an implementation of the base RK method, whose numerical solution conserves the invariant $V$ upto the tolerance of linear equation solver (not the tolerance of nonlinear equation solver); 
    \item provides a simple implementation of arbitrary higher order linearly implicit conservative methods. For example, to achieve $2s$ order based on the Gauss method, we can employ the explicit Euler step to obtain $Y^{(0)}_i$'s, and just iterate the procedure above $ 2 s - 1$ times. 
\end{itemize}
Moreover, in \cref{itr:explicit}, we propose a computationally cheaper version of the iterative procedure, and show the similar theorems. 

\subsection{Convergence of semi-implicit iteration}
\label{itr:conv}

The procedure defined in \cref{def:si} converges linearly as shown below. 
The theorem reveals that, the iterative procedure can be regarded as an implementation of the base RK method. 
There, unlike usual nonlinear equation solvers, 
the numerical solution $ y^{(k)}_1 $ preserves the quadratic invariant $V$ for any $k$. 

\begin{theorem}[Linear convergence of semi-implicit iteration]\label{thm:conv_si}
    Let $ S : \RR^d \to \RR^{d \times d}$ be the Lipschitz continuous map, and let $ \{ Y^{(k)}_i \}_{k,i} $ be the set of vectors satisfying the relation~\eqref{eq:itr_siu}. 
    Assume that $ h < \overline{h} := \paren*{ \norm{A}_{\infty} \norm*{Q}_2 \paren*{ L_S C + M_S } }^{-1}$ holds, where $ C = \max_{i \in [s]} \norm*{ Y_i } $, $ M_S = \max_{ i \in [s] } \paren*{ L_S \norm*{ Y_i^{(0)} - Y_i } + \norm*{ S(Y_i) }_2 }$, and $ Y_i $'s are the inner stages of the base Runge--Kutta method. 
    Then, for each $k = 1,2,\dots$,
    \begin{equation}\label{ineq:conv}
    \max_{i \in [s]} \norm*{ Y^{(k+1)}_i - Y_i } \le \frac{ h \norm*{ A }_{\infty} \norm*{Q}_2 L_S C }{1-h \norm*{ A }_{\infty}  \norm*{Q}_2 M_S } \max_{ i \in [s] } \norm*{ Y^{(k)}_i - Y_i }
    \end{equation}
    holds. 
    In particular, $Y_i^{(k)}$'s converge linearly to $Y_i$'s. 
\end{theorem}

\begin{proof}
    By using $ M_S^{(k)} = \max_{i \in [s] } \norm*{ S \paren[\Big]{Y^{(k)}_i} }_2 $, we see
    \begin{align}
        \max_{i \in [s]} \norm*{ Y^{(k+1)}_i - Y_i }
        &= h \max_{i \in [s]} \norm*{ \sum_{ j \in [s]} a_{ij} \paren*{ S \paren*{ Y^{(k)}_j } Q Y^{ (k+1) }_j - S \paren*{ Y_j } Q Y_j } } \\
        &\le h \norm*{ A }_{\infty} \max_{ j \in [s]} \norm*{ S \paren*{ Y^{(k)}_j } Q Y^{ (k+1) }_j - S \paren*{ Y_j } Q Y_j } \\
        &\le h \norm*{A}_{\infty} \left( \max_{j \in [s]} \norm*{ S \paren*{ Y^{(k)}_j } Q Y^{ (k+1) }_j - S \paren*{ Y^{(k)}_j } Q Y_j } \right. \\
        &\qquad \qquad + \left. \max_{ j \in [s]} \norm*{ S \paren*{ Y^{(k)}_j } Q Y_j - S \paren*{ Y_j } Q Y_j } \right) \\
        &\le h \norm*{A}_{\infty} \norm*{Q}_2 \left( M_S^{(k)} \max_{ i \in [s] } \norm*{ Y^{(k+1)}_i - Y_i } + L_S C \max_{ i \in [s] } \norm*{ Y^{(k)}_i - Y_i }  \right), 
    \end{align}
    which proves that
    \begin{equation}\label{ineq:conv_k}
        \max_{i \in [s]} \norm*{ Y^{(k+1)}_i - Y_i } \le \frac{ h \norm*{ A }_{\infty} \norm*{Q}_2 L_S C }{1-h \norm*{ A }_{\infty} \norm*{Q}_2 M_S^{(k)} } \max_{ i \in [s] } \norm*{ Y^{(k)}_i - Y_i }
    \end{equation}
    holds if $ h < \overline{h}^{(k)} := \paren*{ \norm*{A}_{\infty} \norm*{Q}_2 M_S^{(k)} }^{-1} $. 
    
    To bridge the gap between \eqref{ineq:conv_k} and \eqref{ineq:conv}, it is sufficient to prove $ M_S^{(k)} \le M_S \ ( k = 0 , 1,\dots ) $  
    (notice that $ M_S^{(k)} \le M_s $ implies $ \overline{h} < \overline{h}^{(k)} $). 
    Here, $ M_S^{(k)} \le M_S $ is proved by induction. 
    When $ k = 0$, $ M_S^{(0)} \le M_S $ holds  
    by Lipschitz continuity of $ S $. 
    Then, assuming $ M_S^{(\ell)} \le M_S $ for $ \ell = 0,1,\dots, k $, we prove $ M_S^{(k+1)} \le M_S$. 
    By the induction assumption, we see $ h < \overline{h} <  \overline{h}^{(\ell)} $, which enables us to use the inequality~\eqref{ineq:conv_k}, i.e. for $ \ell = 0,1,\dots ,k $, 
    \begin{equation}
        \max_{ i \in [s] } \norm*{ Y_i^{(\ell+1)} - Y_i } \le \frac{ h \norm*{ A }_{\infty} \norm*{Q}_2 L_S C }{1-h \norm*{ A }_{\infty} \norm*{Q}_2 M_S^{(\ell)} } \max_{ i \in [s] } \norm*{ Y^{(\ell)}_i - Y_i } < \max_{i \in [s] } \norm*{ Y_i^{(\ell)} - Y_i }  
    \end{equation}
    holds, where the last inequality holds because of $ h < \overline{h}$. Therefore, we obtain
    \begin{align}
        M_S^{(k+1)} 
        &\le \max_{ i \in [s] }\paren*{ \norm*{  S \paren*{ Y^{(k+1)}_i } -  S \paren*{ Y_i }  }_2 + \norm*{  S \paren*{ Y_i } }_2 } \\
        &\le \max_{ i \in [s] }\paren*{ L_S \norm*{  Y^{(k+1)}_i -  Y_i  } + \norm*{  S \paren*{ Y_i } }_2 } \\
        &< \max_{ i \in [s] }\paren*{ L_S \norm*{  Y^{(0)}_i -  Y_i  } + \norm*{  S \paren*{ Y_i } }_2 } = M_S, 
    \end{align}
    which completes the induction. 
\end{proof}

\subsection{Accuracy of semi-implicit iterative schemes}
\label{itr:acc}

Here, we prove that, if we iterate the procedure more, 
the resulting numerical solution $ y^{(k)}_1 $ will be more accurate upto the accuracy of the base RK method. 

\begin{theorem}[Accuracy of semi-implicit iteration schemes]\label{thm:err_si_itr}
Assume the following conditions: 
\begin{description}
    \item[(A1)] $ Y^{(0)}_i $ satisfies $ \norm[\big]{ Y^{(0)}_i - y( c_i h ) } \le C h^{q} $ for each $ i \in [s] ${\em ;} 
    \item[(A2)] $ S : \RR^d \to \RR^{d \times d} $ is Lipschitz continuous{\em ;}
    \item[(A3)] the base Runge--Kutta method has order $p${\em .}
\end{description}
Then, the numerical solution $ y^{(k)}_1 $ of the semi-implicit iterative scheme satisfies 
\begin{equation}\label{eq:acc_itr}
    \norm*{ y^{(k)}_1 - y (h) } \le C' h^{ \min\{ p , q + k-1\} +1 }
\end{equation}
for sufficiently small $h > 0$, i.e. the scheme with $k$ iteration has order $ \min\{ p , q + k-1\} $. 
Here, $ C' $ is a constant depending only on the exact solution $y$, $k$, $S$, $Q$, $A$, $b$ and the constant $C$ in (A1).   
\end{theorem}

\begin{proof}
Due to the assumption (A1), there exists a smooth function $ y^{(0)} : [0,h] \to \RR^d $ satisfying $ y^{(0)} (c_i h ) = Y^{(0)}_i $ and $ \sup_{ t \in [0,h ]} \norm{ y^{(0)} (t) - y (t) }  \le C h^{q} $. 
Then, the semi-implicit iterative scheme can be regarded as the usual Runge--Kutta method corresponding to $A, b$ for the system
\begin{equation}\label{eq:acc_system}
    \begin{cases}
        \dot{y}^{(1)} = S \paren*{ y^{(0)} } \nabla V \paren*{ y^{(1)} }, \\[10pt]
        \dot{y}^{(2)} = S \paren*{ y^{(1)} } \nabla V \paren*{ y^{(2)} }, \\
        \qquad \vdots \\
        \dot{y}^{(k)} = S \paren*{ y^{(k-1)} } \nabla V \paren*{ y^{(k)} }. 
    \end{cases}
\end{equation}
Thus, thanks to the assumption (A3), 
it is sufficient to prove $ \norm[\big]{ y^{(k)} (h) - y(h) } \le C^{(k)} h^{q+k} $. 

To this end, 
we prove 
$$ \sup_{ t \in [0,h]} \norm[\big]{ y^{(j)} (t) - y (t) } \le C^{(j)} h^{q+j} \qquad ( j = 1,\dots, k )$$
by induction, where $ C^{(j)} := \paren*{ 2 L_S \norm*{Q}_2 C_y }^{j} C \ ( C_y := \sup_{ t \in [0,h] } \norm*{ y(t) })$. 
We assume that $ \sup_{ t \in [0,h]} \norm[\big]{ y^{(j-1)} (t) - y (t) } \le C^{(j-1)} h^{q+j-1} $ holds so that 
\begin{align} M_S^{(j-1)} 
&:= \sup_{ t \in [0,h] } \norm*{ S\paren*{ y^{(j-1)} (t) } }_2\\
&\le \norm*{ S(y_0) }_2 + L_S \paren[\big]{ C^{(j-1)} h^{q+j-1} + \sup_{t \in [0,h] } \norm*{ y(t) - y_0 } }\\
&< \infty .  
\end{align}
Then, we see
\begin{align}
    \sup_{ t \in [0,h]} \norm*{ y^{(j)} (t) - y (t) } 
    &= \sup_{ t \in [0,h]} \norm*{ \int^t_0 \paren*{ S \paren*{ y^{(j-1)} (r) } Q y^{(j)} (r) - S (y(r)) Q y (r) } \rd r } \\
    &\le \sup_{ t \in [0,h] } \norm*{ \int^t_0 \paren*{ S \paren*{ y^{(j-1)} (r) } Q y^{(j)} (r) - S \paren*{y^{(j-1)}(r)} Q y (r) } \rd r } \\
    &\qquad + \sup_{t \in [0,h]} \norm*{ \int^t_0 \paren*{ S \paren*{ y^{(j-1)} (r) } Q y (r) - S (y(r)) Q y (r) } \rd r } \\
    &\le h M_S^{(j-1)} \norm*{Q}_2 \sup_{ t \in [0,h]} \norm*{ y^{(j)} (t) - y(t) } + L_S C^{(j-1)} \norm*{Q}_2 C_y h^{q+j}. 
\end{align}
Since we assume $ h $ is sufficiently small, in particular, $ h \le ( 2 M_S^{(j-1)} \norm*{Q}_2 )^{-1}$ holds, we see
\[ \sup_{ t \in [0,h]} \norm*{ y^{(j)} (t) - y (t) } \le 2 L_S C^{(j-1)} \norm*{Q}_2 C_y h^{q+j}. \]
Thus, by induction, we obtain $\sup_{ t \in [0,h]} \norm[\big]{ y^{(j)} (t) - y (t) } \le C^{(j)} h^{q+j} \ ( j = 1,\dots, k )$. 
Since it implies $ \norm[\big]{ y^{(k)} (h) - y(h) } \le C^{(k)} h^{q+k} $, 
the theorem holds. 
\end{proof}

\subsection{Explicit iteration}
\label{itr:explicit}

The iterative procedure defined in \cref{def:si} requires a linear equation solver at each iteration. 
Thus, the procedure is a bit expensive. 
Since the semi-implicit update~\eqref{eq:itr_siu} is employed to ensure the conservation law, 
a cheaper procedure can be considered as follows. 

\begin{definition}[Iterative scheme (explicit update)]\label{def:ex}
\hfill
\begin{description}
    \item[Step 0] Prepare $ Y^{(0)}_i \approx y (c_i h) $ and set $k = 1$. 
    \item[Step 1] If some criteria hold, go to Step 2. 
        Otherwise, compute $ Y^{(k)}_i $'s by 
        \begin{equation}\label{eq:itr_eu}
            Y^{(k)}_i = y_0 + h \sum_{ j \in [s]} a_{ij} S \paren*{Y^{(k-1)}_j } \nabla V \paren*{ Y^{(k-1)}_j } \qquad ( i \in [s] ). 
        \end{equation}
        Set $ k = k + 1 $ and repeat Step 1. 
    \item[Step 2] Solve
        \begin{equation}
            Y^{(k)}_i = y_0 + h \sum_{ j \in [s]} a_{ij} S \paren*{Y^{(k-1)}_j } \nabla V \paren*{ Y^{(k)}_j } \qquad ( i \in [s])
        \end{equation}
        to obtain $ Y^{(k)}_i $'s and output 
        \[ y_1^{(k)} = y_0 + h \sum_{j \in [s] } b_j S \paren*{Y^{(k-1)}_j } \nabla V \paren*{ Y^{(k)}_j }. \]
\end{description}
\end{definition}

Even when we employ the explicit iteration above, 
the resulting numerical solution preserves the quadratic invariant thanks to the semi-implicit update in Step~1. 
Moreover, the counterparts of \cref{thm:conv_si,thm:err_si_itr} also hold. 

As well as the semi-implicit iteration, the explicit iteration converges as shown below (the proof is omitted). 

\begin{theorem}[Linear convergence of explicit iteration]\label{thm:conv_ei}
    Let $ S : \RR^d \to \RR^{d \times d}$ be the Lipschitz continuous map, and let $ \{ Y^{(k)}_i \}_{k,i} $ be the set of vectors satisfying the relation~\eqref{eq:itr_eu}. 
    Assume that $ h < \overline{h} := \paren*{ \norm{A}_{\infty} \norm*{Q}_2 \paren*{ L_S C + M_S } }^{-1}$ holds, where $ C = \max_{i \in [s]} \norm*{ Y_i } $, $ M_S = \max_{ i \in [s] } \paren*{ L_S \norm*{ Y_i^{(0)} - Y_i } + \norm*{ S(Y_i) }_2 }$, and $ Y_i $'s are the inner stages of the base Runge--Kutta method. 
    Then, for each $k = 1,2,\dots$,
    \begin{equation}\label{ineq:conv_ei}
    \max_{i \in [s]} \norm*{ Y^{(k+1)}_i - Y_i } \le h \norm{A}_{\infty} \norm*{Q}_2 \paren*{ L_S C + M_S } \max_{ i \in [s] } \norm*{ Y^{(k)}_i - Y_i }
    \end{equation}
    holds. 
    In particular, $Y_i^{(k)}$'s converge linearly to $Y_i$'s. 
\end{theorem}
\begin{remark}
As shown in \cref{thm:conv_si,thm:conv_ei}, the upper bound $ \overline{h} $ for semi-implicit and explicit procedures coincide. 
On the other hand, in view of the rate of convergence, 
the semi-implicit procedure may be superior to the explicit one because 
\begin{equation}
   \frac{ h \norm*{ A }_{\infty} \norm*{Q}_2 L_S C }{1-h \norm*{ A }_{\infty}  \norm*{Q}_2 M_S } < h \norm{A}_{\infty} \norm*{Q}_2 \paren*{ L_S C + M_S }
\end{equation}
holds for any $ h < \overline{h}$. 
\end{remark}

Finally, we show the accuracy of the resulting numerical solution. 
The proof is quite similar to that of \cref{thm:err_si_itr}, but 
we use 
\begin{equation}\label{eq:acc_system_ex}
    \begin{cases}
        \dot{y}^{(1)} = S \paren*{ y^{(0)} } \nabla V \paren*{ y^{(0)} }, \\[10pt]
        \dot{y}^{(2)} = S \paren*{ y^{(1)} } \nabla V \paren*{ y^{(1)} }, \\
        \qquad \vdots \\
        \dot{y}^{(k-1)} = S \paren*{ y^{(k-2)} } \nabla V \paren*{ y^{(k-2)} },\\
        \dot{y}^{(k)} = S \paren*{ y^{(k-1)} } \nabla V \paren*{ y^{(k)}}
    \end{cases}
\end{equation}
instead of \eqref{eq:acc_system}. 

\begin{theorem}[Accuracy of explicit iteration schemes]\label{thm:err_ex_itr}
Assume the following conditions: 
\begin{description}
    \item[(A1)] $ Y^{(0)}_i $ satisfies $ \norm[\big]{ Y^{(0)}_i - y( c_i h ) } \le C h^{q} $ for each $ i \in [s] ${\em ;} 
    \item[(A2)] $ S : \RR^d \to \RR^{d \times d} $ is Lipschitz continuous{\em ;}
    \item[(A3)] the base Runge--Kutta method has order $p${\em .}
\end{description}
Then, the numerical solution $ y^{(k)}_1 $ of the explicit iterative scheme satisfies 
\begin{equation}\label{eq:acc_itr_ex}
    \norm*{ y^{(k)}_1 - y (h) } \le C' h^{ \min\{ p , q + k-1\} +1 }
\end{equation}
for sufficiently small $h > 0$, i.e. the scheme with $k$ iteration has order $ \min\{ p , q + k-1\} $. 
Here, $ C' $ is a constant depending only on the exact solution $y$, $k$, $S$, $Q$, $A$, $b$ and the constant $C$ in (A1).   
\end{theorem}

\section{Numerical Experiments}
\label{sec:ne}

In this section, we numerically examine the proposed schemes. 
We employ the $6$th order Gauss method
\begingroup
\renewcommand{\arraystretch}{1.3}
\renewcommand{\arraycolsep}{4pt}
\[ 
\begin{array}{c|ccc} \frac{1}{2}-\frac{\sqrt{15}}{10} & \frac{5}{36} & \frac{2}{9}-\frac{\sqrt{15}}{15} & \frac{5}{36}-\frac{\sqrt{15}}{30} \\ \frac{1}{2} & \frac{5}{36}+\frac{\sqrt{15}}{24} & \frac{2}{9} & \frac{5}{36}-\frac{\sqrt{15}}{24} \\ \frac{1}{2}+\frac{\sqrt{15}}{10} & \frac{5}{36}+\frac{\sqrt{15}}{30} & \frac{2}{9}+\frac{\sqrt{15}}{15} & \frac{5}{36} \\ \hline & \frac{5}{18} & \frac{4}{9} & \frac{5}{18} \end{array}. \]
\endgroup
as the base Runge--Kutta method.
We compare three methods for preparing $Y^{(0)}_i$'s (see \cref{proposed:general_predictors}): 
\begin{itemize}
    \item $5$th order CERK (Continuous Explicit Runge--Kutta) method. 
    \item Extrapolation using inner stages in the previous step. 
    \item Locally second order approximation by the explicit Euler method ($ Y^{(0)}_i = y_0 + c_i h S(y_0) \nabla V(y_0) $). 
\end{itemize}
Then, \cref{thm:err_si_itr,thm:err_ex_itr} provide us with the accuracy of each method as shown in \cref{tab:nm_ac}. 

\begin{table}[ht]
\caption{Accuracy of each numerical method ensured by \cref{thm:err_si_itr,thm:err_ex_itr}. }
\label{tab:nm_ac}
\centering
\begingroup
\renewcommand{\arraystretch}{1.1}
\renewcommand{\arraycolsep}{4pt}
\begin{tabular}{c|ccccc}
    \hline
                           & \multicolumn{5}{c}{Iteration} \\
                           & $1$ & $2$ & $3$ & $4$ & $>5$ \\ \hline 
    CERK ($q=6$)           & $6$ & $6$ & $6$ & $6$ & $6$ \\
    extrapolation ($q=4$)  & $4$ & $5$ & $6$ & $6$ & $6$ \\
    Euler ($q=2$) & $2$ & $3$ & $4$ & $5$ & $6$ \\ \hline
\end{tabular}
\endgroup
\end{table}

We confirm that each numerical scheme certainly achieves the accuracy given in \cref{tab:nm_ac} by using the Euler equation (\cref{ne:Euler}), the Kepler equation (\cref{ne:Kepler}), and the Korteweg--de Vries equation (\cref{ne:KdV}) as examples. 
In addition, we confirm the conservation law in \cref{ne:Euler}, 
and the long time behaviour in \cref{ne:Kepler}. 

\subsection{Euler equation}
\label{ne:Euler}

First, we consider the Euler equation for rigid body rotation
\begingroup
\renewcommand{\arraystretch}{1.1}
\renewcommand{\arraycolsep}{3pt}
\begin{equation}\label{eq:Euler}
\dot{y} = S(y) \nabla H(y), \quad S(y) = \begin{pmatrix} 0 & \alpha y_3 & -\beta y_2 \\ - \alpha y_3 & 0 & y_1 \\ \beta y_2 & - y_1 & 0 \end{pmatrix}, \quad H (y) = \frac{y_1^2+y_2^2+y_3^2}{2}, 
\end{equation}
\endgroup
where $ \alpha $ and $ \beta $ are parameters. 
We use the initial value $ y_0 = (0,1,1)^{\trans} $ and parameters $ \alpha = 1 + (1/\sqrt{1.51})$, $ \beta = 1 - (0.51/\sqrt{1.51})$, which are employed in \cite{M2015}. The exact solution is periodic with the period $ T = 4 K (0.51) \approx 7.450563209330954 $. 

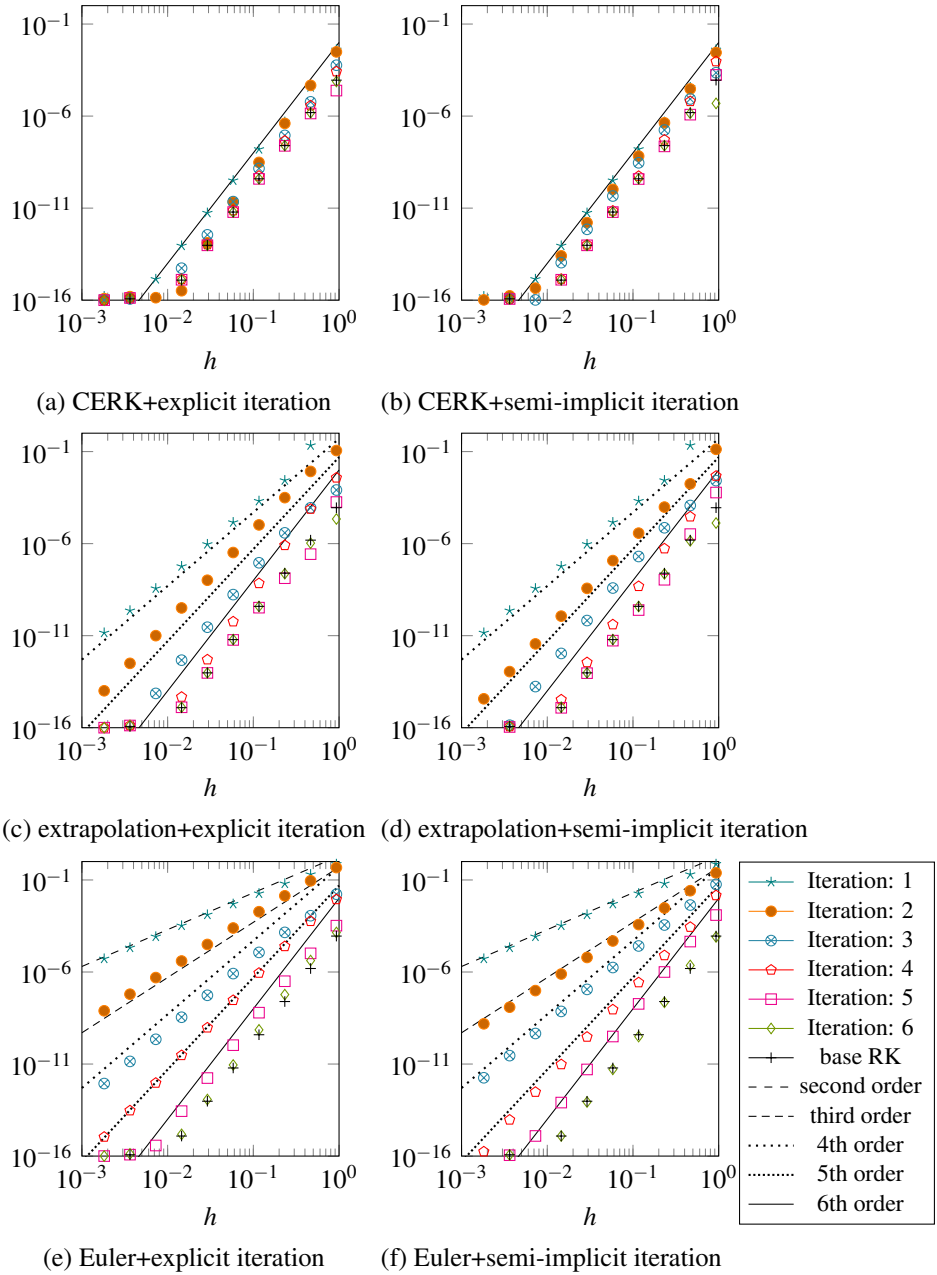
\begin{figure}[thp]
    \begin{minipage}{0.4\textwidth}
    \centering
    
    \begin{tikzpicture}
    \begin{loglogaxis}[compat = newest, width = 5cm, height = 5.5cm, xmin = 1e-3, xmax = 1, ymin = 1e-16, ymax = 1,legend pos=north west, xlabel={$h$},cycle list name=mylist]
        \foreach \x in {1,...,6} {
        \addplot table[y index = \x, only marks] {data/Euler200817/errorsCERKexplicit.dat};
    }
    \addplot[black,mark = +] table [only marks] {data/Euler200817/errorsRK.dat};
    \addplot[black,domain=1e-3:1] {0.01*x^6};
    \end{loglogaxis}
    \end{tikzpicture}
    
    (a) CERK+explicit iteration
    \end{minipage}
    \begin{minipage}{0.59\textwidth}
    \begin{tikzpicture}
    \begin{loglogaxis}[compat = newest, width = 5cm, height = 5.5cm, xmin = 1e-3, xmax = 1, ymin = 1e-16, ymax = 1,legend pos=north west, xlabel={$h$},cycle list name=mylist]
        \foreach \x in {1,...,6} {
        \addplot table[y index = \x, only marks] {data/Euler200817/errorsCERKsemiimplicit.dat};
    }
    \addplot[black,mark = +] table [only marks] {data/Euler200817/errorsRK6.dat};
    \addplot[black,domain=1e-3:1] {0.01*x^6};
    \end{loglogaxis}
    \end{tikzpicture}
    
    (b) CERK+semi-implicit iteration
    \end{minipage}\\[5pt]
    
    \begin{minipage}{0.4\textwidth}
    \centering
    
    \begin{tikzpicture}
    \begin{loglogaxis}[compat = newest, width = 5cm, height = 5.5cm, xmin = 1e-3, xmax = 1, ymin = 1e-16, ymax = 1,legend pos=north west, xlabel={$h$},cycle list name=mylist]
        \foreach \x in {1,...,6} {
        \addplot table[y index = \x, only marks] {data/Euler200817/errorsExexplicit.dat};
    }
    \addplot[black,mark = +] table [only marks] {data/Euler200817/errorsRK6.dat};
    \addplot[black,dotted,thick,domain=1e-3:1] {0.5*x^4};
    \addplot[black,densely dotted,thick,domain=1e-3:1] {0.05*x^5};
    \addplot[black,domain=1e-3:1] {0.01*x^6};
    \end{loglogaxis}
    \end{tikzpicture}
    
    (c) extrapolation+explicit iteration
    \end{minipage}
    \begin{minipage}{0.59\textwidth}
    \begin{tikzpicture}
    \begin{loglogaxis}[compat = newest, width = 5cm, height = 5.5cm, xmin = 1e-3, xmax = 1, ymin = 1e-16, ymax = 1,legend pos=north west, xlabel={$h$},cycle list name=mylist]
        \foreach \x in {1,...,6} {
        \addplot table[y index = \x, only marks] {data/Euler200817/errorsExsemiimplicit.dat};
    }
    \addplot[black,mark = +] table [only marks] {data/Euler200817/errorsRK.dat};
    \addplot[black,dotted,thick,domain=1e-3:1] {0.5*x^4};
    \addplot[black,densely dotted,thick,domain=1e-3:1] {0.05*x^5};
    \addplot[black,domain=1e-3:1] {0.01*x^6};
    \end{loglogaxis}
    \end{tikzpicture}
    
    (d) extrapolation+semi-implicit iteration
    \end{minipage}\\[5pt]
    
    \begin{minipage}{0.4\textwidth}
    \centering
    
    \begin{tikzpicture}
    \begin{loglogaxis}[compat = newest, width = 5cm, height = 5.5cm, xmin = 1e-3, xmax = 1, ymin = 1e-16, ymax = 1, xlabel={$h$},cycle list name=mylist]
        \foreach \x in {1,...,6} {
        \addplot table[y index = \x,only marks] {data/Euler200817/errorsACEX2explicit.dat};
    }
    \addplot[black,mark = +] table [only marks] {data/Euler200817/errorsRK.dat};
    \addplot[black,dashed,domain=1e-3:1] {2*x^2};
    \addplot[black,densely dashed,domain=1e-3:1] {0.5*x^3};
    \addplot[black,dotted,thick,domain=1e-3:1] {0.5*x^4};
    \addplot[black,densely dotted,thick,domain=1e-3:1] {0.05*x^5};
    \addplot[black,domain=1e-3:1] {0.01*x^6};
    \end{loglogaxis}
    \end{tikzpicture}
    
    (e) Euler+explicit iteration
    \end{minipage}
    \begin{minipage}{0.59\textwidth}
            
    \begin{tikzpicture}
    \begin{loglogaxis}[compat = newest, width = 5cm, height = 5.5cm, xmin = 1e-3, xmax = 1, ymin = 1e-16, ymax = 1,xlabel={$h$},cycle list name=mylist,legend style={nodes={scale=0.9, transform shape}, at={(1.46,1)},anchor=north,legend columns=1}]
        \foreach \x in {1,...,6} {
        \edef\temp{\noexpand\addlegendentry{Iteration: \x}}
        \addplot table[y index = \x, only marks] {data/Euler200817/errorsACEX2semiimplicit.dat};
        \temp
    }
    \addplot[black,mark = +] table [only marks] {data/Euler200817/errorsRK6.dat};
    \addlegendentry{base RK};
    \addplot[black,dashed,domain=1e-3:1] {2*x^2};
    \addlegendentry{second order};
    \addplot[black,densely dashed,domain=1e-3:1] {0.5*x^3};
    \addlegendentry{third order};
    \addplot[black,dotted,thick,domain=1e-3:1] {0.5*x^4};
    \addlegendentry{4th order};
    \addplot[black,densely dotted,thick,domain=1e-3:1] {0.05*x^5};
    \addlegendentry{5th order};
    \addplot[black,domain=1e-3:1] {0.01*x^6};
    \addlegendentry{6th order};
    \end{loglogaxis}
    \end{tikzpicture}
    
    (f) Euler+semi-implicit iteration
    \end{minipage}
    
    \caption{Relative errors of numerical solutions for the Euler equation~\eqref{eq:Euler}.}
    \label{fig:eul_errors}
\end{figure}

The relative errors of the numerical solutions are summarised in \cref{fig:eul_errors}. 
They decrease along the reference lines drawn based on \cref{tab:nm_ac}. 
In this case, there are no serious difference between explicit and semi-implicit iterations. 

\begin{figure}[htp]
    \centering
    \includegraphics[scale=0.9]{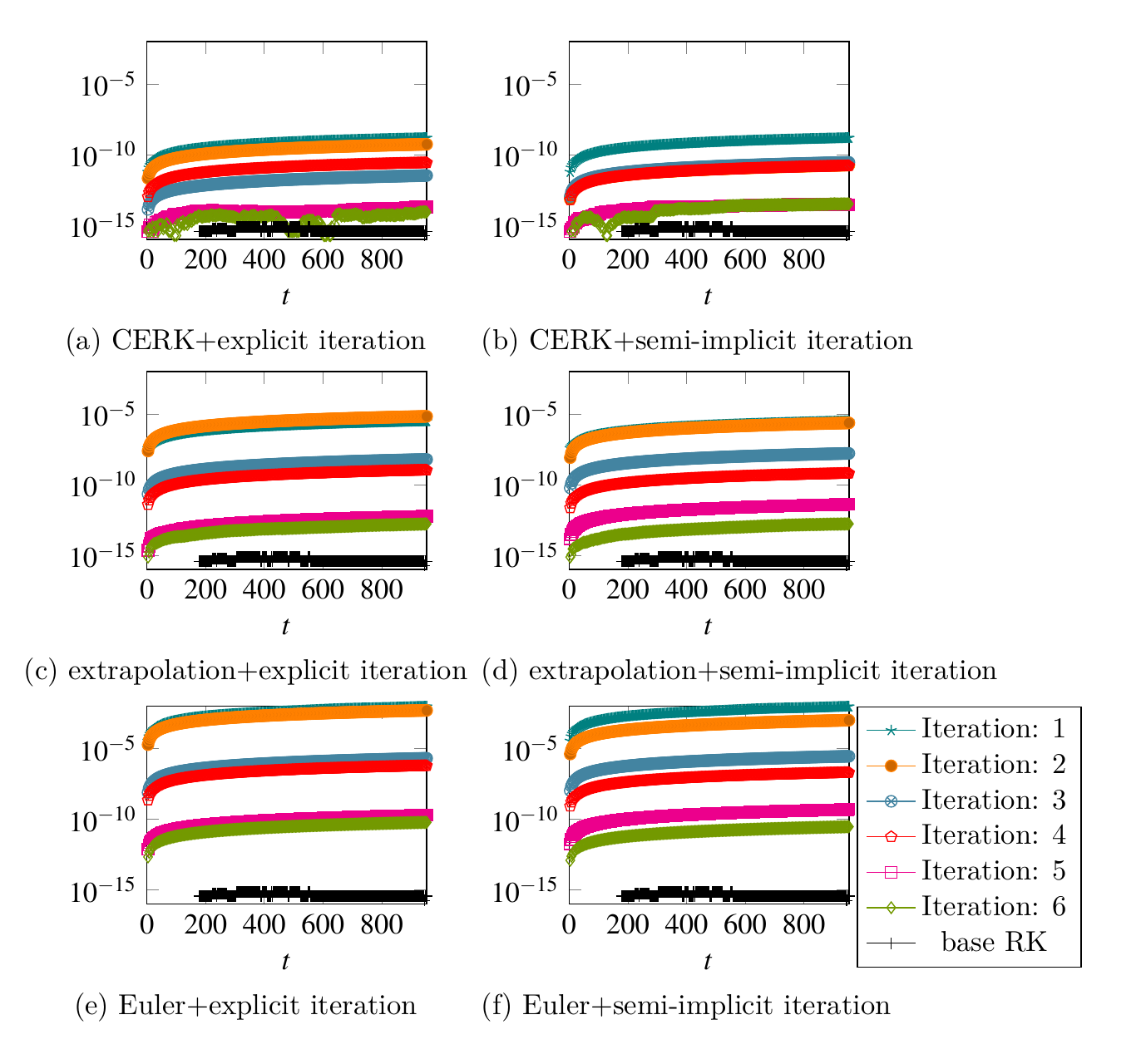}
    \vspace{-10pt}
    
    \caption{Evolution of relative errors of the invariant of $I(y) = (y_1^2 + \beta y_2^2 + \alpha y_3^2 )/2$ of the Euler equation ($ h = 4K(0.51)/128 \approx 0.0582 $). Markers are for every 64 points.}
    \label{fig:eul_I}
\end{figure}

In order to confirm the preservation of the invariant $H$, 
we conducted the numerical experiment with 128 periods and $ h = 4K(0.51)/128 \approx 0.0582$. 
As expected, the relative errors of the invariants are the order of $ 10^{-14} $ even in such long numerical simulations (see \cref{fig:eul_H} in appendix). 

The relative errors of another quadratic invariant $I(y) = (y_1^2 + \beta y_2^2 + \alpha y_3^2 )/2$ are summarised in \cref{fig:eul_I}. 
Note that, the Gauss method, the base RK method, preserves all quadratic invariants. 
On the other hand, the proposed schemes do not exactly preserve $I$. 
However, since the numerical solutions of the proposed schemes converge to that of the Gauss method, 
the relative error of $I$ tends to decrease as the number of iterations increases.

\subsection{Kepler equation}
\label{ne:Kepler}

Next, we consider the Kepler equation
\begingroup
\renewcommand{\arraystretch}{1.1}
\renewcommand{\arraycolsep}{3pt}
\begin{equation}\label{eq:Kepler}
\dot{y} = S(y) \nabla V(y), \quad S(y) = \begin{pmatrix} 0 & -1 & 0 & 0 \\ 1 & 0 & 0 & 0 \\ 0 & 0 & 0 & - \frac{1}{ \paren*{ y_1^2 +y_2^2 }^{3/2}} \\ 0 & 0 & \frac{1}{ \paren*{ y_1^2 +y_2^2 }^{3/2}} & 0 \end{pmatrix}, \quad V (y) = y_1 y_4 - y_2 y_3
\end{equation}
\endgroup
with the initial condition $ y_0 = (1-e,0,0,\sqrt{ (1+e)/(1-e) })^{\trans} $. 
The exact solution is periodic with the period $ T = 2 \pi $.

\begin{figure}[thp]
\pgfplotsset{compat = newest, width = 5cm, height = 5.5cm, xmin=1e-3,xmax=1,ymin = 1e-16, ymax =1,ytick={1e-16,1e-12,1e-8,1e-4,1},xlabel={$h$},cycle list name=mylist}

    \begin{minipage}{0.4\textwidth}
    \centering
    
    \begin{tikzpicture}
    \begin{loglogaxis}
        \foreach \x in {1,...,6} {
        \addplot table[y index = \x, only marks] {data/Kepler/errorsCERK5explicit.dat};
    }
    \addplot[black,mark = +] table [only marks] {data/Kepler/errorsRK6.dat};
    \addplot[black,domain=1e-3:1] {0.1*x^6};
    \end{loglogaxis}
    \end{tikzpicture}
    
    (a) CERK+explicit iteration
    \end{minipage}
    \begin{minipage}{0.59\textwidth}
    \begin{tikzpicture}
    \begin{loglogaxis}
        \foreach \x in {1,...,6} {
        \addplot table[y index = \x, only marks] {data/Kepler/errorsCERK5semiimplicit.dat};
    }
    \addplot[black,mark = +] table [only marks] {data/Kepler/errorsRK6.dat};
    \addplot[black,domain=1e-3:1] {0.1*x^6};
    \end{loglogaxis}
    \end{tikzpicture}
    
    (b) CERK+semi-implicit iteration
    \end{minipage}\\[5pt]
    
    \begin{minipage}{0.4\textwidth}
    \centering
    
    \begin{tikzpicture}
    \begin{loglogaxis}
        \foreach \x in {1,...,6} {
        \addplot table[y index = \x, only marks] {data/Kepler/errorsEX3explicit.dat};
    }
    \addplot[black,mark = +] table [only marks] {data/Kepler/errorsRK6.dat};
    \addplot[black,dotted,thick,domain=1e-3:1] {x^4};
    \addplot[black,thick,densely dotted,domain=1e-3:1] {0.05*x^5};
    \addplot[black,domain=1e-3:1] {0.1*x^6};
    \end{loglogaxis}
    \end{tikzpicture}
    
    (c) extrapolation+explicit iteration
    \end{minipage}
    \begin{minipage}{0.59\textwidth}
    \begin{tikzpicture}
    \begin{loglogaxis}
        \foreach \x in {1,...,6} {
        \addplot table[y index = \x, only marks] {data/Kepler/errorsEX3semiimplicit.dat};
    }
    \addplot[black,mark = +] table [only marks] {data/Kepler/errorsRK6.dat};
    \addplot[black,dotted,thick,domain=1e-3:1] {x^4};
    \addplot[black,domain=1e-3:1] {0.1*x^6};
    \end{loglogaxis}
    \end{tikzpicture}
    
    (d) extrapolation+semi-implicit iteration
    \end{minipage}\\[5pt]
    
    \begin{minipage}{0.4\textwidth}
    \centering
    
    \begin{tikzpicture}
    \begin{loglogaxis}
        \foreach \x in {1,...,6} {
        \addplot table[y index = \x,only marks] {data/Kepler/errorsACEX2explicit.dat};
    }
    \addplot[black,mark = +] table [only marks] {data/Kepler/errorsRK6.dat};
    \addplot[black,dashed,domain=1e-3:1] {10*x^2};
    \addplot[black,densely dashed,domain=1e-3:1] {0.01*x^3};
    \addplot[black,dotted,thick,domain=1e-3:1] {x^4};
    \addplot[black,domain=1e-3:1] {0.1*x^6};
    \end{loglogaxis}
    \end{tikzpicture}
    
    (e) Euler+explicit iteration
    \end{minipage}
    \begin{minipage}{0.59\textwidth}
    
    \begin{tikzpicture}
    \begin{loglogaxis}[legend style={nodes={scale=0.9, transform shape},at={(1.46,1)},anchor=north,legend columns=1}]
        \foreach \x in {1,...,6} {
        \edef\temp{\noexpand\addlegendentry{Iteration: \x}}
        \addplot table[y index = \x, only marks] {data/Kepler/errorsACEX2semiimplicit.dat};
        \temp
    }
    \addplot[black,mark = +] table [only marks] {data/Kepler/errorsRK6.dat};
    \addlegendentry{base RK};
    \addplot[black,dashed,domain=1e-3:1] {10*x^2};
    \addlegendentry{second order};
    \addplot[black,densely dashed,domain=1e-3:1] coordinates{(1e-3,1.1)(1,1.1)};
    \addlegendentry{third order};
    \addplot[black,dotted,thick,domain=1e-3:1] {x^4};
    \addlegendentry{4th order};
    \addplot[black,domain=1e-3:1,densely dotted,thick] coordinates{(1e-3,1.1)(1,1.1)};
    \addlegendentry{5th order};
    \addplot[black,domain=1e-3:1] {0.1*x^6};
    \addlegendentry{6th order};
    \end{loglogaxis}
    \end{tikzpicture}
    
    (f) Euler+semi-implicit iteration
    \end{minipage}
    \caption{Relative errors of numerical solutions for the Kepler equation~\eqref{eq:Kepler} ($ e = 0.01$).}
    \label{fig:kep_errors}
\end{figure}

The numerical results are summarised in \cref{fig:kep_errors}. 
Note that, the skew-symmetric matrix function $ S $ is defined on $ \RR^4 \setminus \setI{ (0,0,y_3,y_4)}{y_3 , y_4 \in \RR} $ and not globally Lipshitz. 
Thus, \cref{thm:err_si_itr,thm:err_ex_itr} cannot be applied directly. 
However, since $S$ is locally Lipschitz, the accuracy of the numerical schemes should coincides with that in \cref{thm:err_si_itr,thm:err_ex_itr}. 
The numerical results of explicit iteration schemes are matched with this observation as shown in \cref{fig:kep_errors} (a), (c), and (e). 
However, the semi-implicit iteration schemes seem to achieve higher order than expected. 
These schemes turn out to achieve the order $ \min\{ p, q + 2k - 2 \} $ (see \cref{ap:proof_Kepler} for the proof). 
This is due to the special structure of the Kepler equation.

\begin{figure}[thp]
\pgfplotsset{compat = newest, width = 5cm, height = 4.5cm, xmin=-2,xmax=1,ymin = -1.5, ymax =1.5,xlabel={$y_1$},ylabel={$y_2$},cycle list name=mylist,restrict x to domain*=-2.1:1.1,restrict y to domain*=-1.6:1.6,layers/my layer set/.define layer set={main,foreground}{},set layers=my layer set}

    \begin{minipage}{0.4\textwidth}
    \centering
    
    \begin{tikzpicture}
    \begin{axis}
        \foreach \x in {1,...,6} {
        \edef\temp{\noexpand\addplot table[x index = 1, y index = 2, only marks] {data/Kepler/orbitCERK5explicit\x.dat}}
        \temp;
    }
    \addplot[black, on layer = foreground] table[x index = 1, y index = 2] {data/Kepler/exact.dat};
    \end{axis}
    \end{tikzpicture}
    
    (a) CERK+explicit iteration
    \end{minipage}
    \begin{minipage}{0.59\textwidth}
    \begin{tikzpicture}
    \begin{axis}
        \foreach \x in {1,...,6} {
        \edef\temp{\noexpand\addplot table[x index = 1, y index = 2, only marks] {data/Kepler/orbitCERK5semiimplicit\x.dat}}
        \temp;
    }
    \addplot[black, on layer = foreground] table [x index = 1, y index = 2] {data/Kepler/exact.dat};
    \end{axis}
    \end{tikzpicture}
    
    (b) CERK+semi-implicit iteration
    \end{minipage}\\[5pt]
    
    \begin{minipage}{0.4\textwidth}
    \centering
    
    \begin{tikzpicture}
    \begin{axis}
        \foreach \x in {1,...,6} {
        \edef\temp{\noexpand\addplot table[x index = 1, y index = 2, only marks] {data/Kepler/orbitEX3explicit\x.dat}}
        \temp;
    }
    \addplot[black, on layer = foreground] table [x index = 1, y index = 2] {data/Kepler/exact.dat};
    \end{axis}
    \end{tikzpicture}
    
    (c) extrapolation+explicit iteration
    \end{minipage}
    \begin{minipage}{0.59\textwidth}
    \begin{tikzpicture}
    \begin{axis}
        \foreach \x in {1,...,6} {
        \edef\temp{\noexpand\addplot table[x index = 1, y index = 2, only marks] {data/Kepler/orbitEX3semiimplicit\x.dat}}
        \temp;
    }
    \addplot[black, on layer = foreground] table [x index = 1, y index = 2] {data/Kepler/exact.dat};
    \end{axis}
    \end{tikzpicture}
    
    (d) extrapolation+semi-implicit iteration
    \end{minipage}\\[5pt]
    
    \begin{minipage}{0.4\textwidth}
    \centering
    
    \begin{tikzpicture}
    \begin{axis}
        \foreach \x in {1,...,6} {
        \edef\temp{\noexpand\addplot table[x index = 1, y index = 2, only marks] {data/Kepler/orbitACEX2explicit\x.dat}}
        \temp;
    }
    \addplot[black, on layer = foreground] table[x index = 1, y index = 2] {data/Kepler/exact.dat};
    \end{axis}
    \end{tikzpicture}
    
    (e) Euler+explicit iteration
    \end{minipage}
    \begin{minipage}{0.59\textwidth}
    
    \begin{tikzpicture}
    \begin{axis}[legend style={at={(1.4,1)},anchor=north,legend columns=1}]
        \foreach \x in {1,...,6} {
        \edef\tempa{\noexpand\addplot table[x index = 1, y index = 2, only marks] {data/Kepler/orbitACEX2semiimplicit\x.dat}}
        \edef\temp{\noexpand\addlegendentry{Iteration: \x}}
        \tempa;
        \temp
    }
    \addplot[black, on layer = foreground] table[x index = 1,y index = 2] {data/Kepler/exact.dat};
    \addlegendentry{exact sol.};
    \end{axis}
    \end{tikzpicture}
    
    (f) Euler+semi-implicit iteration
    \end{minipage}
    \caption{Numerical solutions for the last 1 period of 1024 periods of numerical experiments for the Kepler equation~\eqref{eq:Kepler} ($h = 2 \pi / 64$, $ e = 0.6$). The curves in black are the exact periodic solutions. }
    \label{fig:kep_orbits}
\end{figure}
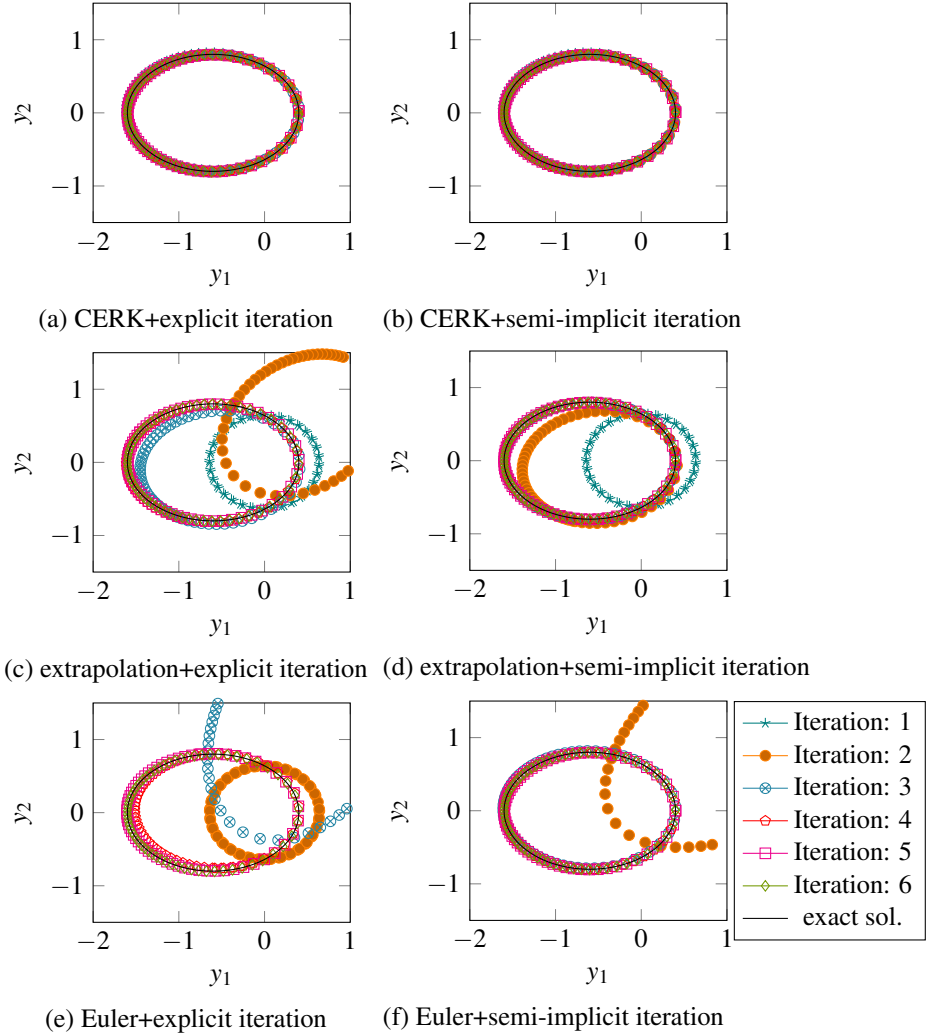

To confirm the long time behaviour of the proposed schemes, we conduct numerical experiments with $ T = 2 \pi \times 1024 $ and $ h = 2 \pi / 64$ ($ e = 0.6 $).  
\Cref{fig:kep_orbits} shows the numerical solutions for the last 1 period of the numerical experiments. 
The numerical solutions of all schemes do not diverge even in such a long run
(the solution in (e) and (f) with one iteration is also finite, but it is far from the exact solution and does not appear in the figure).

\subsection{Korteweg--de Vries equation}
\label{ne:KdV}

Finally, we consider the Korteweg--de Vries (KdV) equation 
\begin{equation}\label{eq:KdV}
    u_t + 6 u u_x + u_{xxx} = 0,
\end{equation}
which has the quadratic invariant $ \frac{1}{2} \int u^2 \, \rd x $. 
Here, $u$ is a dependent variable, $ t $ and $x$ are temporal and spatial independent variables (subscripts denote partial derivatives). 
For brevity, we deal with the periodic domain $ \Sb := \RR / L \ZZ $ and use the cnoidal wave solution
\begin{equation}\label{eq:cnoidal}
    u(t,x) = u_0 + 2 k^2 \kappa^2 \paren*{ \cn \paren*{ \kappa ( x - x_0 - ct ) | k} }^2,
\end{equation}
where $ c = 6 u_0 + 4 (2 k^2 -1) \kappa^2 $, $L = 4 K (k) / \kappa $, and $ T = \abs*{ \frac{L}{c \kappa} }$. 
In this section, we use the parameters $ k = \sqrt{0.1}$, $u_0 = 0$, and $ \kappa = 1 $
 ($L \approx 3.2248826974404383$, $ T \approx 1.007775842950137$). 

Then, we use the norm-preserving Fourier-spectral discretisation ($ \Delta x := L/d $, $ y_k (t) \approx u ( t, k \Delta x) $ ($k =  1, 2, \dots , d$))
\begin{equation}\label{eq:kdv_np}
\dot{y} = S(y) \nabla V(y), \quad V(y) = \frac{1}{2} \sum_{k=1}^d y_k^2 \Delta x, 
\end{equation}
where the skew-symmetric matrix function $ S $ is defined as 
\[ \paren*{ S(v) w }_k := - 2 \paren*{ v_k \delta_x w_k + \delta_x \paren*{ \paren*{ v_k w_k } } } - \delta_x^3 w_k. \]
Here, $ \delta_x $ denotes the Fourier-spectral difference operator (see, e.g. \cite{Fornberg1996}). 

Due to the stiffness of the ODE~\eqref{eq:kdv_np}, 
the direct application of the CERK method does not work well in this case. 
A good approach to deal with such a stiff ODE is to employ the exponential integrators (see, e.g. \cite{SL2019}). 
The ODE~\eqref{eq:kdv_np} can be rewritten as
\[ \dot{y} = A y + g (y), \qquad A = - \delta_x^3, \qquad  \paren*{ g (y) }_k = -2 \paren*{ y_k \delta_x y_k + \delta_x \paren*{ y_k^2 } }. \]
Then, we transform it by using $ \widetilde{y} (t):= e^{ - t A } y (t) $ and obtain
\[ \dot{ \widetilde{y} } (t) = e^{-tA} g \paren*{ e^{tA} \widetilde{y} }. \]
Since the transformed ODE is numerically easier to deal with, 
we apply the CERK method to it. 

The numerical results are summarised in \cref{fig:kdv_errors} ($ d = 16$), 
which are basically matched with \cref{tab:nm_ac}. 
However, in contrast to the previous two examples, 
there is a large difference between the explicit and semi-implicit iteration: 
the explicit iteration requires very small step sizes.

\begin{figure}[thp]
\pgfplotsset{compat = newest, width = 5cm, height = 5.5cm, xmin=1e-5,xmax=0.3,ymin = 1e-16, ymax =1,ytick={1e-16,1e-12,1e-8,1e-4,1},xlabel={$h$},cycle list name=mylist}

    \begin{minipage}{0.4\textwidth}
    \centering
    
    \begin{tikzpicture}
    \begin{loglogaxis}
        \foreach \x in {1,...,6} {
        \addplot table[y index = \x, only marks] {data/KdVtol20/errorsCExpRK5explicit.dat};
    }
    \addplot[black,mark = +] table [only marks] {data/KdVtol20/errorsRK6.dat};
    \addplot[black,domain=1e-4:0.3] {0.5*x^6};
    \end{loglogaxis}
    \end{tikzpicture}
    
    (a) CERK+explicit iteration
    \end{minipage}
    \begin{minipage}{0.59\textwidth}
    \begin{tikzpicture}
    \begin{loglogaxis}
        \foreach \x in {1,...,6} {
        \addplot table[y index = \x, only marks] {data/KdVtol20/errorsCExpRK5semiimplicit.dat};
    }
    \addplot[black,mark = +] table [only marks] {data/KdVtol20/errorsRK6.dat};
    \addplot[black,domain=1e-4:0.3] {0.5*x^6};
    \end{loglogaxis}
    \end{tikzpicture}
    
    (b) CERK+semi-implicit iteration
    \end{minipage}\\[5pt]
    
    \begin{minipage}{0.4\textwidth}
    \centering
    
    \begin{tikzpicture}
    \begin{loglogaxis}
        \foreach \x in {1,...,6} {
        \addplot table[y index = \x, only marks] {data/KdVtol20/errorsEX3explicit.dat};
    }
    \addplot[black,mark = +] table [only marks] {data/KdVtol20/errorsRK6.dat};
    \addplot[black,dotted,thick,domain=1e-4:0.3] {0.5*x^4};
    \addplot[black,densely dotted,thick,domain=1e-4:0.3] {0.5*x^5};
    \addplot[black,domain=1e-4:0.3] {0.5*x^6};
    \end{loglogaxis}
    \end{tikzpicture}
    
    (c) extrapolation+explicit iteration
    \end{minipage}
    \begin{minipage}{0.59\textwidth}
    \begin{tikzpicture}
    \begin{loglogaxis}
        \foreach \x in {1,...,6} {
        \addplot table[y index = \x, only marks] {data/KdVtol20/errorsEX3semiimplicit.dat};
    }
    \addplot[black,mark = +] table [only marks] {data/KdVtol20/errorsRK6.dat};
    \addplot[black,dotted,thick,domain=1e-4:0.3] {0.5*x^4};
    \addplot[black,densely dotted,thick,domain=1e-4:0.3] {0.5*x^5};
    \addplot[black,domain=1e-4:0.3] {0.5*x^6};
    \end{loglogaxis}
    \end{tikzpicture}
    
    (d) extrapolation+semi-implicit iteration
    \end{minipage}\\[5pt]
    
    \begin{minipage}{0.4\textwidth}
    \centering
    
    \begin{tikzpicture}
    \begin{loglogaxis}
        \foreach \x in {1,...,6} {
        \addplot table[y index = \x,only marks] {data/KdVtol20/errorsACEX2explicit.dat};
    }
    \addplot[black,mark = +] table [only marks] {data/KdVtol20/errorsRK6.dat};
    \addplot[black,dashed,domain=1e-4:0.3] {0.3*x^2};
    \addplot[black,densely dashed,domain=1e-4:0.3] {0.1*x^3};
    \addplot[black,dotted,thick,domain=1e-4:0.3] {0.5*x^4};
    \addplot[black,densely dotted,thick,domain=1e-4:0.3] {0.5*x^5};
    \addplot[black,domain=1e-4:0.3] {0.5*x^6};
    \end{loglogaxis}
    \end{tikzpicture}
    
    (e) Euler+explicit iteration
    \end{minipage}
    \begin{minipage}{0.59\textwidth}
    
    \begin{tikzpicture}
    \begin{loglogaxis}[legend style={nodes={scale=0.9, transform shape},at={(1.46,1)},anchor=north,legend columns=1}]
        \foreach \x in {1,...,6} {
        \edef\temp{\noexpand\addlegendentry{Iteration: \x}}
        \addplot table[y index = \x, only marks] {data/KdVtol20/errorsACEX2semiimplicit.dat};
        \temp
    }
    \addplot[black,mark = +] table [only marks] {data/KdVtol20/errorsRK6.dat};
    \addlegendentry{base RK};
    \addplot[black,dashed,domain=1e-4:0.3] {0.3*x^2};
    \addlegendentry{second order};
    \addplot[black,densely dashed,domain=1e-4:0.3] {0.1*x^3};
    \addlegendentry{third order};
    \addplot[black,dotted,thick,domain=1e-4:0.3] {0.5*x^4};
    \addlegendentry{4th order};
    \addplot[black,densely dotted,thick,domain=1e-4:0.3] {0.5*x^5};
    \addlegendentry{5th order};
    \addplot[black,domain=1e-4:0.3] {0.5*x^6};
    \addlegendentry{6th order};
    \end{loglogaxis}
    \end{tikzpicture}
    
    (f) Euler+semi-implicit iteration
    \end{minipage}
    \caption{Relative errors of numerical solutions for the KdV equation~\eqref{eq:KdV} ($ d = 16$). }
    \label{fig:kdv_errors}
\end{figure}
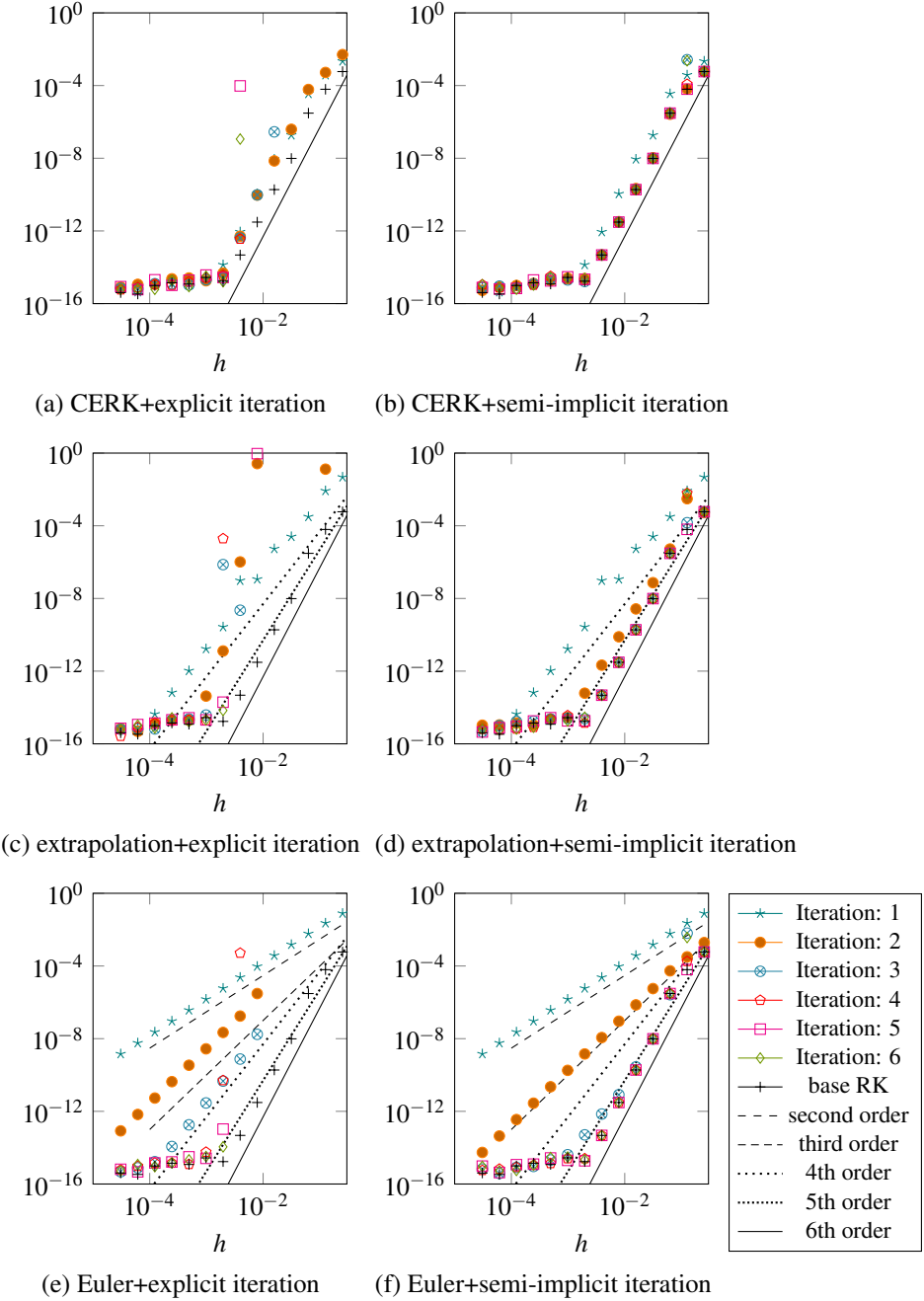

\bibliographystyle{plain}
\bibliography{reference}

\begin{appendix}
\metrics{2}{0.8}{2}{0.8}

\section{Proof of orders of PRK schemes in \cref{proposed:PRK}}
\label{ap:proof_PRK}

Here, we prove that the PRK scheme defined by~\eqref{eq:PRK_Gauss2} has order $4$. 
In view of \cref{thm:PRK_OC}, 
it is suffice to check that the condition $ \phi (\tau) = 1/ \tau ! $ holds for any tree in the following set
\[ \setE*{\raisebox{-3mm}{$
\black\tree{},
\tree{\vertex10},
\tree{\white\vertex10},
\tree{\vertex1{0.6}\vertex2{-0.6}},
\tree{\vertex1{0.6}\change\vertex2{-0.6}},
\tree{\white\vertex1{0.6}\vertex2{-0.6}},
\tree{\vertex10\vertex10},
\tree{\vertex10\change\vertex10},
\tree{\change\vertex10\vertex10},
\tree{\vertex1{0.875}\vertex20\vertex3{-0.875}},
\tree{\vertex1{0.875}\vertex20\change\vertex3{-0.875}},
\tree{\vertex1{0.875}\change\vertex20\vertex3{-0.875}},
\tree{\change\vertex1{0.875}\vertex20\vertex3{-0.875}},
\tree{\vertex1{0.6}\vertex10\vertex3{-0.6}},
\tree{\vertex1{0.6}\change\vertex10\change\vertex3{-0.6}},
\tree{\change\vertex1{0.6}\vertex10\change\vertex3{-0.6}},
\tree{\vertex1{0.6}\vertex10\change\vertex3{-0.6}},
\tree{\vertex1{0.6}\change\vertex10\vertex3{-0.6}},
\tree{\change\vertex1{0.6}\vertex10\vertex3{-0.6}},
\tree{\vertex10\vertex1{0.6}\vertex2{-0.6}},
\tree{\vertex10\vertex1{0.6}\change\vertex2{-0.6}},
\tree{\vertex10\change\vertex1{0.6}\vertex2{-0.6}},
\tree{\change\vertex10\vertex1{0.6}\vertex2{-0.6}},
\tree{\vertex10\vertex10\vertex10},
\tree{\vertex10\vertex10\change\vertex10},
\tree{\vertex10\change\vertex10\vertex10},
\tree{\change\vertex10\vertex10\vertex10}$}
} \]
Since $ c_i = \widehat{c}_i $ and $ \sum_{ j \in [s]} a_{ij} c_j = \sum_{ j \in [s] } \widehat{a}_{ij} \widehat{c}_j $ hold for $ i = 4 , 5 $ (indices of nonzero elements of $A$ and $b$), 
these conditions are satisfied except for that corresponding to \tree{\change\vertex10\vertex1{0.6}\vertex2{-0.6}} and \tree{\change\vertex10\vertex10\vertex10}, i.e. 
$ \sum_{i,j \in [s] } b_i \widehat{a}_{ij} \widehat{c}_j^2 = 1/12$ and $ \sum_{ i , j , k \in [s]} b_i \widehat{a}_{ij} \widehat{a}_{jk} \widehat{c}_k =1/24$. 
These conditions can also be verified by a simple calculation. 

We also prove that the PRK scheme defined by~\eqref{eq:PRK_433dic} has order $3$. 
In view of \cref{thm:PRK_OC}, we should confirm the conditions corresponding to 
\[ \setE*{\raisebox{-1.8mm}{$
\black\tree{},
\tree{\vertex10},
\tree{\white\vertex10},
\tree{\vertex1{0.6}\vertex2{-0.6}},
\tree{\vertex1{0.6}\change\vertex2{-0.6}},
\tree{\white\vertex1{0.6}\vertex2{-0.6}},
\tree{\vertex10\vertex10},
\tree{\vertex10\change\vertex10},
\tree{\change\vertex10\vertex10}$}
} \]
Thanks to $ c_i = \widehat{c}_i $, these conditions are satisfied except for that corresponding to \tree{\change\vertex10\vertex10}, i.e. $ \sum_{ i, j \in [s]} b_i \widehat{a}_{ij} \widehat{c}_j = 1/6 $ which is equivalent to~\eqref{eq:PRK_433dic_cond}. 

\section{Numerical results}

Here, we show the figure omitted in \cref{ne:Euler}. 
The relative errors of the invariant $H (y) = (y_1^2 + y_2^2 + y_3^2 )/2$ of the Euler equation are summarised in \cref{fig:eul_H}. 

\begin{figure}[ht]
    \centering
    \includegraphics[scale=0.9]{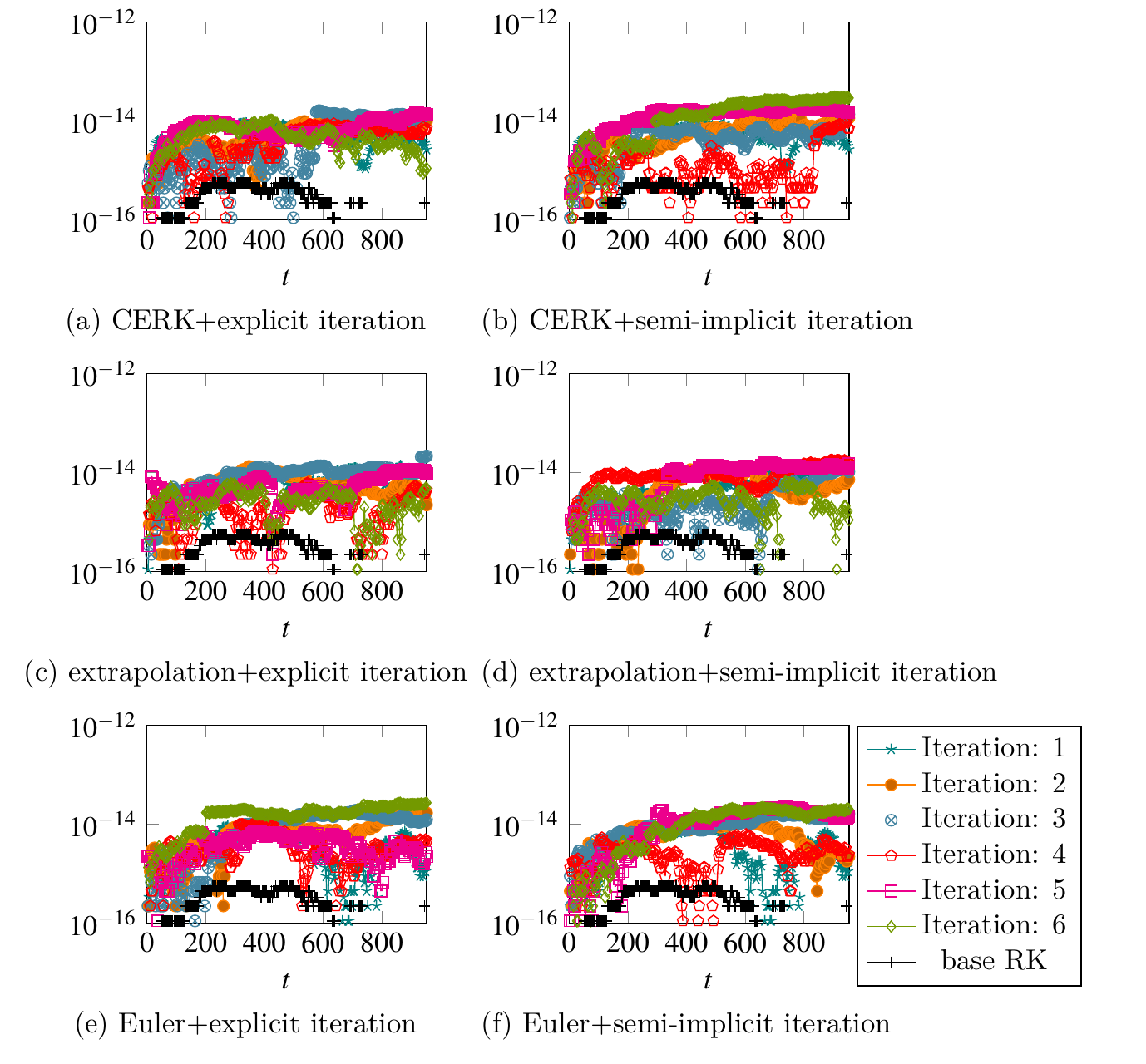}
    \vspace{-10pt}
    
    \caption{Evolution of relative errors of the invariant $H$ of the Euler equation ($ h = 4K(0.51)/128 \approx 0.0582 $). Markers are for every 64 points.}
    \label{fig:eul_H}
\end{figure}

\section{Accuracy of the proposed schemes for the Kepler equation}
\label{ap:proof_Kepler}

As shown in \cref{fig:kep_errors}, the semi-implicit iteration scheme has higher orders than expected. 
In this section, we prove the order for the Kepler case. 

\begin{proposition}\label{prop:err_si_itr_kep}
Assume the following conditions: 
\begin{description}
    \item[(A1)] $ Y^{(0)}_i $ satisfies $ \norm[\big]{ Y^{(0)}_i - y( c_i h ) } \le C h^{q} $ for each $ i \in [s] ${\em ;} 
    \item[(A2)] the exact solution $ y $ satisfies $ \underline{R} \le \sqrt{ \paren*{ y_1 (t) }^2 + \paren*{ y_2 (t) }^2 } \le \overline{R} $ for any $ t \in [0,h] $ for some $\underline{R}, \overline{R} > 0${\em ;}
    \item[(A3)] the Runge--Kutta method corresponding to $ A, b $ has order $p${\em .}
\end{description}
Then, the numerical solution $ y^{(k)}_1 $ of the semi-implicit iterative scheme for the Kepler equation~\eqref{eq:Kepler} satisfies 
\begin{equation}\label{eq:acc_itr_kep}
    \norm*{ y^{(k)}_1 - y (h) } \le C' h^{ \min\{ p , q + 2k-2\} +1 }
\end{equation}
for sufficiently small $h > 0$, i.e. the scheme with $k$ iteration have order $ \min\{ p , q + 2k-2\} $. 
Here, $ C' $ is a constant depending only on $k$, $ \overline{R}$,  $\underline{R} $, $A$, $b$ and the constant $C$ in (A1).   
\end{proposition}

\begin{proof}
Similar to the proof of \cref{thm:err_si_itr}, it is sufficient to prove $ \norm*{ y^{(k)} (h) - y (h) } \le C' h^{q+2k-1} $ for the solution $ y^{(k)} $'s of \eqref{eq:acc_system}. 

For each $ j \in [k] $, we see 
\begin{align}\label{eq:kep_sub}
\dot{y}^{(j)}_1 &= y^{(j)}_3, &
\dot{y}^{(j)}_2 &= y^{(j)}_4, &
\dot{y}^{(j)}_3 &= - \frac{1}{ \paren*{ r^{(j-1)} }^3 } y^{(j)}_1, &
\dot{y}^{(j)}_4 &= - \frac{1}{ \paren*{ r^{(j-1)} }^3 } y^{(j)}_2,
\end{align}
where $ r^{(j-1)} = \sqrt{ \paren*{ y^{(j-1)}_1}^2 + \paren*{ y^{(j-1)}_2}^2 } $. 
By introducing $ z^{(j)} = \begin{pmatrix} y^{(j)}_1 & y^{(j)}_2 \end{pmatrix}^{\trans} $ and $ z = \begin{pmatrix} y_1 & y_2 \end{pmatrix}^{\trans} $ for brevity, we see that 
the equation above can be rewritten as
\begin{equation}
    \ddot{ z }^{(j)} = - \frac{1}{ \norm*{ z^{(j-1)} }^3 } z^{(j)}. 
\end{equation}
Then, we prove $ \sup_{t \in [0,h] } \norm*{ z^{(j)} (t) - z(t) } \le C^{(j)} h^{q + 2 j } $ by induction, where 
\begin{align}
    C^{(j)} &= 2 \overline{R} \frac{ \paren*{\overline{R}^{(j-1)} }^2 + \overline{R}^{(j-1)} \overline{R} + \overline{R}^2 }{  \paren*{ \underline{R}^{(j-1)} \underline{R} }^3 } C^{(j-1)}, &
    C^{(0)} &= C, \qquad 
    \overline{R}^{(0)} = \overline{R},\qquad
    \underline{R}^{(0)} = \underline{R},\\
    \overline{R}^{(j)} &= \sup_{t \in [0,h]} \norm*{ z^{(j)} (t) } \le \overline{R} + C^{(j)} h^{q+2j},&
    \underline{R}^{(j)} &= \inf_{t \in [0,h]} \norm*{ z^{(j)} (t) } \ge \underline{R} - C^{(j)} h^{q+2j}
\end{align}
(we assume $h$ is sufficiently small so that $ \paren*{\underline{R}^{(j-1)}}^3 > 2 h^2 $ and $ \underline{R}^{(j)} > 0 $ hold).
To this end, we see
\begin{align}
    \sup_{ t \in [0,h] } \norm*{ z^{(j)} (t) - z(t) }
    &= \sup_{ t \in [0,h] } \norm*{ \int^t_0 \int^{\tau}_0 \paren*{ - \frac{1}{\norm*{ z^{(j-1)} (\sigma) }^3 } z^{(j)} (\sigma) + \frac{1}{\norm*{ z(\sigma) }^3 } z (\sigma)   } \rd \sigma \rd \tau } \\
    &\le \sup_{ t \in [0,h] } \norm*{ \int^t_0 \int^{\tau}_0 \paren*{ - \frac{1}{\norm*{ z^{(j-1)} (\sigma) }^3 } z^{(j)} (\sigma) + \frac{1}{\norm*{ z^{(j-1)}(\sigma) }^3 } z (\sigma)   } \rd \sigma \rd \tau } \\
    &\qquad + \sup_{ t \in [0,h] } \norm*{ \int^t_0 \int^{\tau}_0 \paren*{ - \frac{1}{\norm*{ z^{(j-1)} (\sigma) }^3 } z (\sigma) + \frac{1}{\norm*{ z(\sigma) }^3 } z (\sigma) } \rd \sigma \rd \tau } \\
    &\le \frac{ h^2 }{\paren*{ \underline{R}^{(j-1)} }^3 } \sup_{ t \in [0,h] } \norm*{ z^{(j)} (t) - z(t) }\\
    &\qquad + h^2 \overline{R} \frac{  \paren*{\overline{R}^{(j-1)} }^2 + \overline{R}^{(j-1)} \overline{R} + \overline{R}^2  }{ \paren*{ \underline{R}^{(j-1)} \underline{R} }^3 } \sup_{ t \in [0,h] } \norm*{ z^{(j-1)} (t) - z(t) }, 
\end{align}
which yields
\begin{equation}
    \sup_{ t \in [0,h] } \norm*{ z^{(j)} (t) - z(t) } \le h^2 \frac{ \overline{R} \paren*{ \paren*{\overline{R}^{(j-1)} }^2 + \overline{R}^{(j-1)} \overline{R} + \overline{R}^2 } }{ \underline{R}^3 \paren*{ \paren*{ \underline{R}^{(j-1)} }^3 - h^2 } } \sup_{ t \in [0,h] } \norm*{ z^{(j-1)} (t) - z(t) }. 
\end{equation}
This inequality implies $ \sup_{t \in [0,h] } \norm*{ z^{(j)} (t) - z(t) } \le C^{(j)} h^{q + 2 j } $ by the induction assumption and $ \paren*{\underline{R}^{(j-1)}}^3 > 2 h^2 $. 

Finally, by using \eqref{eq:kep_sub}, we see 
\begin{align}
    \sup_{t \in [0,h]} \norm*{ y^{(k)} (t) - y(t) }^2 
    &\le \sup_{ t \in [0,h] } \norm*{ z^{(k)} (t) - z(t) }^2\\
    &\qquad + \sup_{ t \in [0,h] } \norm*{ \int^t_0 \paren*{ - \frac{1}{\norm*{ z^{(k-1)} (\tau) }^3 } z^{(k)} (\tau) + \frac{1}{\norm*{z(\tau)}^3} z(\tau) } \rd \tau }^2 \\
    &\le \paren*{ C^{(k)} h^{q+2k} }^2 + \paren*{ \frac{h}{\paren*{ \underline{R}^{(k-1)} }^3} C^{(k)} h^{q+2k} + \frac{h}{2} C^{(k)} h^{q+2(k-1)} }^2,
\end{align}
which implies $ \norm*{ y^{(k)} (h) - y (h) } \le C' h^{q+2k-1} $. 
\end{proof}

\end{appendix}
\end{document}